\documentclass[11pt,twoside]{amsart}
\usepackage{amssymb,amsmath,amsthm,mathtools,bm,mathrsfs}
\usepackage{indentfirst}
\usepackage{url}
\usepackage[noadjust]{cite}
\usepackage{microtype}
\usepackage{enumitem}
\usepackage{xcolor}

\allowdisplaybreaks[4]

\usepackage[colorlinks=true,hypertexnames=false,
linkcolor=blue,anchorcolor=red,citecolor=blue,urlcolor=black,
linktocpage,pdfpagelabels,bookmarksnumbered,bookmarksopen]{hyperref}
\numberwithin{equation}{section}

\newtheorem{theorem}{Theorem}[section]
\newtheorem{proposition}[theorem]{Proposition}
\newtheorem{lemma}[theorem]{Lemma}
\newtheorem{corollary}[theorem]{Corollary}
\newtheorem{conjecture}[theorem]{Conjecture}
\theoremstyle{definition}

\theoremstyle{remark}
\newtheorem{remark}[theorem]{Remark}
\newtheorem{question}[theorem]{Question}

\newcommand{\curl}{\nabla\times}
\newcommand{\curlname}{\mathrm{curl}}
\DeclareMathOperator{\diver}{div}
\DeclareMathOperator{\dist}{dist}
\DeclareMathOperator{\sech}{sech}
\newcommand{\R}{\mathbb{R}}
\newcommand{\Hh}{\mathbb{H}}
\newcommand{\Tclass}{\mathcal{T}}
\newcommand{\Sclass}{\mathcal{S}}
\newcommand{\Wspace}{\mathcal{W}}
\newcommand{\Mman}{\mathcal{M}}
\newcommand{\dd}{\,\mathrm{d}}
\newcommand{\e}{\mathrm{e}}
\newcommand{\norm}[2]{\left\lVert #1\right\rVert_{#2}}
\newcommand{\abs}[1]{\left|#1\right|}
\newcommand{\ip}[2]{\left\langle #1,#2\right\rangle}
\newcommand{\Dcurlhalf}{\mathcal{D}^{1,3/2}(\curlname;\R^3)}
\newcommand{\eps}{\varepsilon}

\title[Critical curl extremals and the hyperbolic plane]{Critical curl extremals and the hyperbolic plane}
\author[J. Mederski]{Jarosław Mederski}
\address{J. Mederski\newline\indent
Institute of Mathematics, Polish Academy of Sciences, \newline\indent
ul.~Śniadeckich 8, 00-656 Warsaw, Poland, \newline\indent
and Faculty of Mathematics and Computer Science, \newline\indent
Nicolaus Copernicus University, ul. Chopina 12/18, 87-100 Toruń, Poland}
\email{\href{mailto:jmederski@impan.pl}{jmederski@impan.pl}}
\subjclass[2020]{Primary 35J92; Secondary 35A23, 35Q60, 35R01, 46E35}
\keywords{Sobolev--curl inequality, critical curl--curl equation,
hyperbolic plane, nonlinear duality, semi-explicit extremals,
gauge-canonical inverse iteration, sharp constants, Dirac zero modes}
\date{}
\hypersetup{pdfauthor={Jarosław Mederski},
pdftitle={Critical curl extremals and the hyperbolic plane},
pdfsubject={Critical curl inequalities and the hyperbolic plane},
pdfkeywords={Sobolev-curl inequality, critical curl-curl equation,
hyperbolic plane, nonlinear duality, semi-explicit extremals,
gauge-canonical inverse iteration, sharp constants, Dirac zero modes}}

\begin{document}
\begin{abstract}
We study the critical Sobolev--curl quotient on $\R^3$ at the
conformally invariant exponent $p=3/2$.  Its sharp constant and
optimizers were posed as open problems in the literature.  We identify
the toroidal and poloidal extremals.  Two structural facts are used.
First, the critical curl--curl equation is self-dual: the map
$u\mapsto|\nabla\times u|^{-1/2}\nabla\times u$ is an involution on
finite-energy solutions.  It is the self-dual case of a duality between
the exponents $p$ and $3p/(4p-3)$, which exchanges toroidal and poloidal
fields and relates the sharp constants.  Second, the toroidal minimization
problem reduces exactly to a subcritical $3/2$-Laplace problem on the
hyperbolic plane.  Combined with symmetry and shooting results, these
facts show that the two restricted sharp constants coincide.  They also
classify the gauge-minimal extremals: these are unique up to scaling and
the Euclidean transformations realizing hyperbolic isometries,
semi-explicit, and dual to each other.  For nonzero one-signed profiles
radial about a fixed hyperbolic centre, a gauge-canonical inverse of the
duality yields an iteration by one-dimensional quadratures: the quotients
are nonincreasing and converge to the restricted sharp constant, while
the normalized iterates converge to the extremal pair.  The
first step from an explicit competitor gives the new upper bound
$96\sqrt{15}\pi/(5\sqrt{4455\pi-13568})\approx11.2948$ for the
restricted constants.  The unrestricted sharp constant remains open.
As an application, bounded axisymmetric poloidal potentials admit no
$L^2$ zero mode of the Weyl--Dirac operator.  An explicit conformal
transformation extends this exclusion to both extremals at every real
coupling.
\end{abstract}
\maketitle

\section{Introduction}\label{sec:introduction}

The Sobolev inequality controls a function in terms of its gradient.
For vector fields, replacing the gradient by the curl introduces an
additional difficulty: the kernel of $\curl$ is infinite-dimensional,
since it contains all gradients.  A Sobolev--curl inequality can
therefore control a field only modulo this kernel, and the denominator
of the Sobolev quotient has to be minimized over a gauge class.  For
$p=2$ such a gauge-corrected inequality was first established by the
author and Szulkin \cite{MederskiSzulkin2021}.

We focus on the exponent $p=3/2$, at which the quotient is
conformally invariant, as in the work of Frank and Loss
\cite{FrankLossOptimizers2022,FrankLossSharp2026,FrankLossZeroMode2022}.
For vector fields $u:\R^3\to\R^3$ we consider the quotient
\begin{equation}\label{eq:field-quotient}
Q(u)=\frac{\displaystyle\int_{\R^3}|\curl u|^{3/2}\dd x}
{\displaystyle\left(\inf_{\curl w=0,\ w\in L^3(\R^3,\R^3)}
	\int_{\R^3}|u+w|^3\dd x\right)^{1/2}}
\ge S_{3/2,\curlname}>0,
\end{equation}
where $S_{3/2,\curlname}$ is the sharp Sobolev--curl constant and the
inequality holds for every $u\in L^3(\R^3;\R^3)$ with
$\curl u\in L^{3/2}(\R^3;\R^3)\setminus\{0\}$; see
\cite{FrankLossOptimizers2022,FrankLossSharp2026,MederskiSzulkin2025,FrankLossZeroMode2022}.
Fields in $L^3(\R^3,\R^3)$ whose distributional curl belongs to
$L^{3/2}(\R^3,\R^3)$ are called \emph{finite-energy} fields.

Symmetry classes play a central role in what follows.  In cylindrical
coordinates $x=(\rho\cos\theta,\rho\sin\theta,z)$, $\rho>0$, we use the
Euclidean orthonormal frame
\[
 (e_\rho,e_\theta,e_3)
 =\bigl((\cos\theta,\sin\theta,0),\,
        (-\sin\theta,\cos\theta,0),\,
        (0,0,1)\bigr),
\]
pointing in the radial, azimuthal and axial directions.  Let $\Tclass$
and $\Sclass$ denote, respectively, the \emph{toroidal} and
\emph{poloidal} classes of finite-energy fields of the form
\[
 A(x)=a(\rho,z)e_\theta,
 \qquad
 U(x)=b(\rho,z)e_\rho+c(\rho,z)e_3
\]
for scalar functions $a,b,c$.  The restricted sharp constants are
\[
 S_{3/2,\curlname}^{\Tclass}
 :=\inf_{\substack{A\in\Tclass\\\curl A\ne0}}Q(A),
 \qquad
 S_{3/2,\curlname}^{\Sclass}
 :=\inf_{\substack{U\in\Sclass\\\curl U\ne0}}Q(U),
\]
and, since both classes are contained in the unrestricted admissible
class,
\[
 0<S_{3/2,\curlname}
 \le \min\bigl\{S_{3/2,\curlname}^{\Tclass},
                 S_{3/2,\curlname}^{\Sclass}\bigr\}.
\]
Whether equality holds is the analogue of the ground-state symmetry
question raised for $p=2$ in \cite[Section~7, (P3)]{MederskiSzulkin2021}.

Frank and Loss posed the problem of determining the sharp constant and
the optimizing fields in the critical vector inequality
\cite[Section~9, pp.~24--25]{FrankLossZeroMode2022}; see also
\cite[p.~311, equation~(8)]{FrankLossSharp2026}.
Their attainment theorem \cite[Theorem~1.1]{FrankLossOptimizers2022}
does not identify the optimizing fields.  The author and Szulkin
\cite{MederskiSzulkin2025} proved attainment for general $1<p<3$, also
within the classes $\Tclass$ and $\Sclass$.  An explicit upper bound is
furnished by the field
\[
 A_{\Tclass}(x)=\frac{3}{(1+|x|^2)^2}(-x_2,x_1,0),
\]
a combination of the Loss--Yau potential $A_{\mathrm{LY}}$
\cite{LossYau1986} and its reflection.  Here
$A_{\mathrm{LY}}=A_{\mathrm{LY},e_3}$ belongs to the family
\[
 A_{\mathrm{LY},w}(x)
 =\frac{3}{(1+|x|^2)^2}
 \bigl((1-|x|^2)w+2(w\cdot x)x+2w\times x\bigr),
 \qquad w\in\R^3\setminus\{0\},
\]
which solves equation
\eqref{eq:intro-PDE} below for $|w|=4/3$, as proved by the author and Szulkin
\cite[Theorem~1.3(d)]{MederskiSzulkin2025}; cf.\
\cite[Section~9, (9.4)--(9.5)]{FrankLossZeroMode2022}.
Bieganowski and the author
\cite[Theorem~1.1 and Section~6]{BieganowskiMederski2026} computed
\[
 S_{3/2,\curlname}^{\Tclass}
 \le Q(A_{\Tclass})
 =\frac{8\sqrt{10}}7\pi
 <4\pi=Q(A_{\mathrm{LY}}),
\]
which gives a negative answer to the axisymmetric minimization question
posed by Wang and Zhang \cite[Section~6.2]{WangZhangConstants2026}.
The sharp helicity inequality of Wang and Zhang
\cite{WangZhangSharp2026} has a different denominator and does not
determine the gauge-minimized constant considered here.

We identify the toroidal and poloidal extremals of the critical
Sobolev--curl quotient through an exact hyperbolic representation, and
construct them by a convergent gauge-canonical inverse iteration.  We do
not determine the unrestricted sharp constant: whether the restricted
extremals are global minimizers remains open, and
Section~\ref{sec:open-problems} states the corresponding conjecture
together with the available evidence.

\subsection{Two structural strategies of this work}

A minimizer, taken in its $L^3$-minimal gauge and multiplied by a
suitable positive constant, satisfies the Euler--Lagrange equation
\begin{equation}\label{eq:intro-PDE}
\curl\bigl(|\curl u|^{-1/2}\curl u\bigr)=|u|u.
\end{equation}
The first fact is that \eqref{eq:intro-PDE} is \emph{self-dual}: if
$u$ is a finite-energy weak solution, then its dual field
\[
\mathfrak D(u):=|\curl u|^{-1/2}\curl u
\]
is also a finite-energy weak solution, and
$\mathfrak D(\mathfrak D(u))=u$; the corresponding pair of dual
equations already appears in \cite[Section~9]{FrankLossZeroMode2022}.  Thus
$\mathfrak D$ is an involution on the set of finite-energy weak
solutions; the new feature observed here is that it exchanges toroidal
and poloidal solutions.
Section~\ref{sec:framework} shows that this is the self-dual case of a
duality between the critical curl--curl equations with exponents $p$ and
$\widehat p=3p/(4p-3)$, valid for all $1<p<3$, which interchanges the
two cylindrical classes and relates the sharp constants by
$S_{p,\curlname}^{1/p}=S_{\widehat p,\curlname}^{1/\widehat p}$
(Theorems~\ref{thm:general-p-duality}
and~\ref{thm:general-p-constant-duality}).  At $p=2$ it yields the new
sharp inequality
$\|v\|_2\le S_{2,\curlname}^{-1/2}\|\curl v\|_{6/5}$ for
divergence-free fields, with optimizers obtained from $p=2$ ground
states by taking curl and multiplying by a nonzero constant
(Corollary~\ref{cor:p-two-dual}).

The second new observation concerns the toroidal class.  We write
$A(\psi)=\rho^{-1}\psi(\rho,z)e_\theta$ for the toroidal field
associated with a scalar function $\psi$ on the half-plane
$\{(\rho,z):\rho>0\}$, and we regard this half-plane as the half-plane
model of $\Hh^2$, with metric and area element
\[
 g=\frac{\dd\rho^2+\dd z^2}{\rho^2},
 \qquad
 \dd V_g=\frac{\dd\rho\,\dd z}{\rho^2},
 \qquad
 |\nabla_g\psi|_g=\rho\,|\nabla\psi|.
\]
Integration in the azimuthal variable gives the exact identities
\[
 \int_{\R^3}|A(\psi)|^3\dd x
 =2\pi\int_{\Hh^2}|\psi|^3\dd V_g,
 \qquad
 \int_{\R^3}|\curl A(\psi)|^{3/2}\dd x
 =2\pi\int_{\Hh^2}|\nabla_g\psi|_g^{3/2}\dd V_g,
\]
which Proposition~\ref{prop:T-H-correspondence} extends to the full
energy spaces.  For a general exponent the same computation produces
the weights $\rho^{3-2p}$ and $\rho^{3-p^*}$ in front of the hyperbolic
integrands; both are trivial exactly when $p=3/2$.  The {\em toroidal critical}
quotient is therefore $\sqrt{2\pi}$ times the quotient defining the
{\em subcritical problem}
\begin{equation}\label{eq:Lambda-H}
 \Lambda_{\Hh^2}
 :=\inf_{\psi\in\mathcal D^{1,3/2}(\Hh^2)\setminus\{0\}}
 \frac{\displaystyle\int_{\Hh^2}|\nabla_g\psi|_g^{3/2}\dd V_g}
 {\displaystyle\left(\int_{\Hh^2}|\psi|^3\dd V_g\right)^{1/2}}
 >0,
\end{equation}
where $\mathcal D^{1,3/2}(\Hh^2)$ is the completion of
$\mathcal C_c^\infty(\Hh^2)$ in the norm $\|\nabla_g\psi\|_{L^{3/2}(\Hh^2)}$;
the exponent $3$ lies below the critical exponent $(3/2)^*_2=6$.
In the notation of \cite[(1.2)]{DuttaSandeep2025} the constant is
$S_{0,3}$ with $n=2$ and gradient exponent $3/2$.  We do not know a
closed formula for it.  Writing
$\Delta_{3/2,g}\psi:=\diver_g(|\nabla_g\psi|_g^{-1/2}\nabla_g\psi)$, a
positive minimizer for $\Lambda_{\Hh^2}$, after multiplication by a
suitable positive constant, satisfies
\begin{equation}\label{eq:hyperbolic-PDE}
 -\Delta_{3/2,g}\psi=\psi^2
 \qquad\text{in }\Hh^2
\end{equation}
in the weak sense.  By \cite[Theorem~1.1]{DuttaSandeep2025} every
positive finite-energy solution is radial about some centre
$O\in\Hh^2$, and Theorem~\ref{thm:hyperbolic-profile-uniqueness} below
identifies it, via a weighted shooting classification, as the unique
rapidly decaying member of a one-parameter family: all positive
finite-energy solutions have the same radial profile $\psi_*$, with
only the centre varying.

\subsection{Main results}

Hyperbolic isometries moving the centre correspond to dilations and
axial translations of $\R^3$.  For $\lambda>0$ and $z_0\in\R$ put
\begin{equation}\label{eq:euclidean-moves}
 (T_{\lambda,z_0}u)(x):=\lambda\,u\bigl(\lambda(x-z_0e_3)\bigr),
\end{equation}
the natural action on one-forms; $Q$, equation \eqref{eq:intro-PDE}
and the two classes are invariant under these maps, and
$T_{\lambda,z_0}$ moves a toroidal field with profile radial about
$(1,0)$ to one with profile radial about $(1/\lambda,z_0)$.
To state the iteration, let $\Gamma(u)$ be the unique representative
of the gauge class of $u$ with minimal $L^3$ norm, and let $I(u)$ be
the $L^3$-minimal field whose curl is $|\Gamma(u)|\Gamma(u)$;
equivalently,
\[
 I(u)=\Gamma\!\left(\curl(-\Delta)^{-1}
                 (|\Gamma(u)|\Gamma(u))\right).
\]
The source is divergence-free, because the Euler--Lagrange condition
for gauge minimization is $\diver(|\Gamma(u)|\Gamma(u))=0$.  Thus
$\mathfrak D(I(u))=\Gamma(u)$: on solutions $I$ agrees with
$\mathfrak D$, whereas for general fields with nonzero curl it satisfies
$Q(I(u))\le Q(u)$ (Lemma~\ref{lem:canonical-inverse-descent}).

The first main result reads as follows.
\begin{theorem}\label{thm:main-curl}
Let $\psi_*$ be the positive minimizer for $\Lambda_{\Hh^2}$ centred at
$(1,0)$ and normalized to solve \eqref{eq:hyperbolic-PDE}, and put
\[
 A_*:=\rho^{-1}\psi_*e_\theta,
 \qquad U_*:=\mathfrak D(A_*).
\]
\begin{enumerate}[label=\textup{(\alph*)},leftmargin=*]
\item \textup{(Sharp constants.)}
$0<S_{3/2,\curlname}\le S_{3/2,\curlname}^{\Tclass}
 =S_{3/2,\curlname}^{\Sclass}=\sqrt{2\pi}\,\Lambda_{\Hh^2}$.
\item \textup{(Classification of the extremals.)}
$A_*\in\Tclass$ and $U_*\in\Sclass$ are extremals of the respective
restricted constants, solve \eqref{eq:intro-PDE}, and satisfy
$\mathfrak D(U_*)=A_*$.  Every extremal of
$S_{3/2,\curlname}^{\Tclass}$ is of the form $cT_{\lambda,z_0}A_*$, and
every extremal of $S_{3/2,\curlname}^{\Sclass}$ is of the form
$cT_{\lambda,z_0}U_*+w$, with $c\in\R\setminus\{0\}$, $\lambda>0$,
$z_0\in\R$ and $w\in\Sclass$ with $\curl w=0$.
\item \textup{(Convergence to the two-cycle.)}
Let $\psi\in\mathcal D^{1,3/2}(\Hh^2)$ be nonzero, of one sign and
radial about a centre $O'=(1/\lambda,z_0)$, and let $u_0=A(\psi)$,
$u_{n+1}=I(u_n)$, $v_n=u_n/\|u_n\|_{L^3}$.  Then $Q(u_n)$ is
nonincreasing, $Q(u_n)\to S_{3/2,\curlname}^{\Tclass}$, and
\[
 v_{2k}\longrightarrow\sigma\,\frac{T_{\lambda,z_0}A_*}{\|A_*\|_{L^3}},
 \qquad
 v_{2k+1}\longrightarrow\sigma\,\frac{T_{\lambda,z_0}U_*}{\|U_*\|_{L^3}}
 \qquad\text{in }L^3(\R^3;\R^3),
\]
\[
 \curl v_{2k}\longrightarrow\sigma\,
 \frac{\curl(T_{\lambda,z_0}A_*)}{\|A_*\|_{L^3}},
 \qquad
 \curl v_{2k+1}\longrightarrow\sigma\,
 \frac{\curl(T_{\lambda,z_0}U_*)}{\|U_*\|_{L^3}}
 \qquad\text{in }L^{3/2}(\R^3;\R^3)
\]
as $k\to\infty$, where $\sigma=1$ if $\psi\ge0$ and $\sigma=-1$ if
$\psi\le0$.
\end{enumerate}
\end{theorem}

The iteration is computable: for a nonnegative radial profile, one
full cycle $\Tclass\to\Sclass\to\Tclass$ amounts to two one-dimensional
quadratures, $I^2(A(\psi))=A(\mathcal R\psi)$ with $\mathcal R$ given in
\eqref{eq:inverse-radial-quadrature}, and each iterate is an admissible
competitor, so that every $Q(u_n)$ is a rigorous upper bound for the
restricted constant.  The first step from $A_{\Tclass}$ can be
evaluated in closed form.

\begin{corollary}\label{cor:explicit-bound}
\[
 S_{3/2,\curlname}^{\Tclass}
 <Q(I(A_{\Tclass}))
 =\frac{96\sqrt{15}\,\pi}{5\sqrt{4455\pi-13568}}
 <\frac{8\sqrt{10}}7\pi=Q(A_{\Tclass}).
\]
\end{corollary}

Numerically, $S_{3/2,\curlname}^{\Tclass}\approx11.2615$
(Section~\ref{sec:numerics}), while the two bounds in
Corollary~\ref{cor:explicit-bound} are $11.2948$ and $11.3538$.
The inverse $I$ is iterated rather than $\mathfrak D$, although the
two maps agree on solutions: by
Proposition~\ref{prop:forward-breakdown}, the forward orbit
$\mathfrak D^j(A_{\Tclass})$ has increasing quotients and leaves the
energy space at $j=3$, whereas $I$ lowers the quotient at every step.

The proof of Theorem~\ref{thm:main-curl} is given in
Section~\ref{sec:proof-main}.  Part~\textup{(a)} combines the
hyperbolic representation with the duality; part~\textup{(b)} rests on
the uniqueness of the positive hyperbolic profile
(Theorem~\ref{thm:hyperbolic-profile-uniqueness}), which we derive from
the symmetry theorem of Dutta and Sandeep \cite{DuttaSandeep2025} and
the weighted shooting classification of Kawano, Yanagida and Yotsutani
\cite{KawanoYanagidaYotsutani1993}; part~\textup{(c)} uses the exact
descent formula for $I$, compactness at a fixed centre, and again the
uniqueness.  The profile $\psi_*$ is smooth away from its centre but
only $\mathcal C^{2,1}_{\mathrm{loc}}$ there
(Remark~\ref{rem:profile-regularity}), so the extremals are not smooth
across the circle $\{\rho=1,\ z=0\}$.

The reduction also determines the geometry of the poloidal extremal.
With $s=1+|x|^2$ and $\eta=2\rho/s$, Theorem~\ref{thm:semi-explicit}
gives
\[
 U_*(x)=s^{-2}G_*(\eta)N(x),
 \qquad N(x)=(2x_1z,\,2x_2z,\,1+z^2-\rho^2),
 \qquad
 G_*(\eta)=\frac{2\mathcal J_{\psi_*}(t)}{1-\eta^2},
\]
where $t=\operatorname{arcosh}(1/\eta)$ and
$\mathcal J_{\psi_*}(t)=\int_0^t\psi_*(s)^2\sinh s\dd s$ is the radial
flux of $\psi_*$.  The formula is \emph{semi-explicit}: $N$ and $\eta$
are explicit, and the only remaining factor is computed from the
one-dimensional profile; $G_*$ extends continuously to $[0,1]$ with
$G_*(0)=2\mathcal J_{\psi_*}(\infty)$ and $G_*(1)=\psi_*(0)^2$.

\subsection{Dirac zero modes}

The constant $S_{3/2,\curlname}$ is also relevant to magnetic Weyl--Dirac
operators $\mathcal D_A=\boldsymbol\sigma\cdot(-i\nabla-A)$, where
$\boldsymbol\sigma$ denotes the Pauli matrices.  A \emph{zero mode} is
a nonzero spinor $\Psi\in L^2(\R^3;\mathbb C^2)$ satisfying
$\mathcal D_A\Psi=0$.  The classical Loss--Yau potential
$A_{\mathrm{LY}}$ admits an explicit $L^2$ zero mode, so
$\ker_{L^2}\mathcal D_{A_{\mathrm{LY}}}\ne\{0\}$
\cite[Section~II]{LossYau1986}.  In contrast, we prove that our
semi-explicit extremals have trivial Dirac kernels at every coupling:
$\ker_{L^2}\mathcal D_{cA_*}=\ker_{L^2}\mathcal D_{cU_*}=\{0\}$ for all
$c\in\R$; see Theorem~\ref{thm:extremals-no-zero}.
Frank and Loss \cite[Theorem~1 and Remark~2]{FrankLossSharp2026} showed
that a zero mode requires
\begin{equation}\label{eq:zeromodecon}
\inf_{\nabla\chi\in L^3(\R^3;\R^3)}
\norm{A+\nabla\chi}{L^3}^2\ge3S_{\mathrm{Sob},2},
\qquad S_{\mathrm{Sob},2}=3(\pi/2)^{4/3},
\end{equation}
a sharp condition on the potential modulo gauge.  Combined with the
Sobolev--curl inequality
$\norm{\curl A}{L^{3/2}}\ge S_{3/2,\curlname}^{2/3}
 \inf_{\nabla\chi}\norm{A+\nabla\chi}{L^3}$, it gives
$\norm{\curl A}{L^{3/2}}\ge S_{3/2,\curlname}^{2/3}\sqrt{3S_{\mathrm{Sob},2}}$
for a zero mode, which is strictly below the sharp field-norm threshold
$\norm{\curl A}{L^{3/2}}\ge4S_{\mathrm{Sob},2}$ of Wang and Zhang
\cite[Theorem~1.6]{WangZhangSharp2026}, stated on $\mathbb S^3$ and
transferred to $\R^3$ by conformal invariance, because
$S_{3/2,\curlname}<4\pi$; see also \cite{WangZhangZeroModes2026} and,
for the related spinorial Sobolev inequalities on $\mathbb S^n$ and
their stability, \cite{WangZhangStability2026}.
Their bound doubles the earlier field bound $2S_{\mathrm{Sob},2}$ of
\cite[Theorem~1.1]{FrankLossZeroMode2022} and complements, rather than
replaces, the gauge-invariant condition \eqref{eq:zeromodecon} on the
potential.  Such norm thresholds say nothing at large coupling, where
the following geometric exclusion takes over.

\begin{theorem}\label{thm:extremals-no-zero}
\begin{enumerate}[label=\textup{(\alph*)},leftmargin=*]
\item Let $B=b_\rho(\rho,z)e_\rho+b_z(\rho,z)e_3\in L^\infty(\R^3;\R^3)$
be a bounded axisymmetric poloidal potential.  Then
$\ker_{L^2}\mathcal D_{cB}=\{0\}$ for every $c\in\R$.
\item Let $A_*$ and $U_*$ be the extremals in
Theorem~\ref{thm:main-curl}.  For every $c\in\R$, $\lambda>0$ and
$z_0\in\R$,
\[
 \ker_{L^2}\mathcal D_{cT_{\lambda,z_0}A_*}
 =\ker_{L^2}\mathcal D_{cT_{\lambda,z_0}U_*}=\{0\}.
\]
\end{enumerate}
The operators have domain $H^1(\R^3;\mathbb C^2)$ in
$L^2(\R^3;\mathbb C^2)$.
\end{theorem}

Both extremals are in the minimizing gauge, so the Nehari identity
gives
$\inf_\chi\|cB+\nabla\chi\|_{L^3}^2
 =c^2(S_{3/2,\curlname}^{\Tclass})^{4/3}$ for $B\in\{A_*,U_*\}$, and
\eqref{eq:zeromodecon} excludes zero modes only for
$|c|<\sqrt{3S_{\mathrm{Sob},2}}/(S_{3/2,\curlname}^{\Tclass})^{2/3}
 \approx0.81$.  Theorem~\ref{thm:extremals-no-zero} excludes them at
every coupling: an arbitrarily large gauge-minimal norm does not
ensure a zero mode.  Elton \cite[Theorem~1.1]{Elton2018} proved a
geometric exclusion of this kind for smooth potentials whose magnetic
field is parallel to a conformal Killing field $X$ with
$X\cdot\curl X=0$; part~\textup{(a)} is a version of his angular
obstruction for bounded measurable poloidal potentials, which is what
the non-smooth extremals require, and part~\textup{(b)} reduces $A_*$
to part~\textup{(a)} by an explicit M\"obius map
(Section~\ref{subsec:no-zero}).

Fr\"ohlich, Lieb and Loss \cite{FrohlichLiebLoss1986} showed that
zero modes with finite magnetic-field energy lead to instability
of the one-electron atom in a self-generated magnetic field at
sufficiently large nuclear charge.  The absence of zero modes
established in Theorem~\ref{thm:extremals-no-zero} may therefore
provide a starting point for investigating stability within the
corresponding classes of magnetic potentials.

\subsection{Notation and organization}

Unless another centre is specified, $t$ denotes the hyperbolic
distance from $(\rho,z)=(1,0)$, so that $\eta=\sech t$.
The operators $\curl$ and $\diver$ act in the sense
of distributions.
For positive functions $f,h$, we write $f\asymp h$ when $f/h$ is
bounded above and below by positive constants in the specified limit,
and $f\sim h$ when $f/h\to1$.

Section~\ref{sec:framework} proves the general duality.
Section~\ref{sec:hyperbolic} develops the hyperbolic problem and proves
Theorem~\ref{thm:hyperbolic-profile-uniqueness}.
Section~\ref{sec:reduction} establishes the exact toroidal reduction.
Section~\ref{sec:duality} derives the semi-explicit poloidal extremal
and proves Theorem~\ref{thm:extremals-no-zero}.
Section~\ref{sec:iteration} defines the gauge-canonical inverse
and proves its descent property.
Section~\ref{sec:proof-main} contains the proof of
Theorem~\ref{thm:main-curl} and Corollary~\ref{cor:explicit-bound},
and the breakdown of the forward iteration.
Numerical values and open problems conclude the paper.

\section{Variational framework}\label{sec:framework}

In this section we recall the setting of \cite{MederskiSzulkin2025} and
fix the conventions used later.  For $r>1$, write $r'=r/(r-1)$
for its conjugate exponent.  Let $1<p<3$ and
\[
 p^*:=\frac{3p}{3-p}.
\]
The critical $p$-curl--curl equation is
\begin{equation}\label{eq:p-curl-curl}
 \curl\bigl(\abs{\curl u}^{p-2}\curl u\bigr)
 =\abs{u}^{p^*-2}u
 \qquad\text{in }\R^3.
\end{equation}
The natural space is
\[
 \mathcal D^{1,p}(\curlname;\R^3)
 :=\bigl\{u\in L^{p^*}(\R^3,\R^3):\curl u\in L^p(\R^3,\R^3)\bigr\},
\]
and the kernel of curl is denoted by
\[
 \Wspace:=\bigl\{w\in\mathcal D^{1,p}(\curlname;\R^3):\curl w=0\bigr\}.
\]
The gauge-corrected Sobolev--curl constant is
\begin{equation}\label{eq:global-constant}
 S_{p,\curlname}:=
 \inf_{\curl u\ne0}
 \frac{\displaystyle\int_{\R^3}\abs{\curl u}^{p}\dd x}
 {\displaystyle
  \left(
   \inf_{w\in\Wspace}\int_{\R^3}\abs{u+w}^{p^*}\dd x
  \right)^{p/p^*}}.
\end{equation}
Without this minimization, adding arbitrarily large curl-free fields
to a fixed field with nonzero curl would force the quotient to zero.

For every $u\in\mathcal D^{1,p}(\curlname;\R^3)$ there is a unique
$w(u)\in\Wspace$ satisfying
\begin{equation}\label{eq:gauge-minimizer}
 \int_{\R^3}\abs{u+w(u)}^{p^*}\dd x
 =\inf_{w\in\Wspace}\int_{\R^3}\abs{u+w}^{p^*}\dd x.
\end{equation}
By strict convexity, this is equivalent to
\begin{equation}\label{eq:gauge-EL}
 \int_{\R^3}
 \ip{\abs{u+w(u)}^{p^*-2}(u+w(u))}{\zeta}\dd x=0
 \qquad\text{for every }\zeta\in\Wspace,
\end{equation}
and hence to the nonlinear Coulomb-type gauge condition
\begin{equation}\label{eq:nonlinear-gauge}
 \diver\bigl(\abs{u+w(u)}^{p^*-2}(u+w(u))\bigr)=0
 \quad\text{in }\mathscr D'(\R^3).
\end{equation}
Conversely, \eqref{eq:nonlinear-gauge} implies \eqref{eq:gauge-EL}:
for $1<r<\infty$, every curl-free $Z\in L^r(\R^3,\R^3)$ is the $L^r$
limit of gradients $\nabla\varphi_n$ with
$\varphi_n\in\mathcal C_c^\infty(\R^3)$
\cite[Section~2.1]{MederskiSzulkin2021}, so that
$\int_{\R^3}\ip{B}{Z}\dd x=0$ for every divergence-free
$B\in L^{r'}(\R^3,\R^3)$.  We refer to this fact as the annihilator
identity; it will be used again in the duality arguments below.
Condition \eqref{eq:nonlinear-gauge} defines the set of gauge-minimal fields
used in \cite{MederskiSzulkin2025},
\[
 \Mman:=\bigl\{u\in\mathcal D^{1,p}(\curlname;\R^3):
 \diver(\abs{u}^{p^*-2}u)=0\bigr\},
\]
which contains exactly one representative of every gauge class.

The functional associated with \eqref{eq:p-curl-curl} is
\[
 J_p(u):=\frac1p\int_{\R^3}\abs{\curl u}^{p}\dd x
 -\frac1{p^*}\int_{\R^3}\abs{u}^{p^*}\dd x.
\]
By the main theorem of \cite{MederskiSzulkin2025}, $S_{p,\curlname}$ is attained
and \eqref{eq:p-curl-curl} has a ground state for every $1<p<3$; the same
holds under the cylindrical symmetry restrictions recalled below.

\subsection{The cylindrical symmetry classes}

Write
\[
 x=(\rho\cos\theta,\rho\sin\theta,z),\qquad
 \rho=(x_1^2+x_2^2)^{1/2},\qquad z=x_3,
\]
and let $e_\rho,e_\theta,e_3$ be the cylindrical orthonormal frame.  An
$SO(2)$-equivariant field can be decomposed into radial, azimuthal and axial
parts.

The two fixed-point classes introduced in \cite{MederskiSzulkin2025},
viewed here as subspaces of $\mathcal D^{1,p}(\curlname;\R^3)$, are:
\begin{align}
 \Tclass:
 \quad&A(x)=a(\rho,z)e_\theta
 =\frac{a(\rho,z)}{\rho}(-x_2,x_1,0),
 \notag\\
 \Sclass:
 \quad&U(x)=b(\rho,z)e_\rho+c(\rho,z)e_3
 \notag\\
 &\hspace{15mm}=\alpha(\rho,z)(x_1,x_2,0)+\beta(\rho,z)(0,0,1).
 \label{eq:S-form}
\end{align}
Let $S_{p,\curlname}^{\Tclass}$ and $S_{p,\curlname}^{\Sclass}$ denote the constants
obtained by restricting \eqref{eq:global-constant} to these classes.
By \cite[Theorem~1.4]{MederskiSzulkin2025} both restricted constants are
attained.  After multiplication by a suitable positive constant, the gauge-minimal
representatives of the minimizers solve the full equation
\eqref{eq:p-curl-curl}, not only its restriction to the symmetry subspace.

Since both classes are contained in the full energy space,
\begin{equation}\label{eq:constant-chain}
 0<S_{p,\curlname}
 \le \min\bigl\{S_{p,\curlname}^{\Tclass},
                 S_{p,\curlname}^{\Sclass}\bigr\},
 \qquad 1<p<3.
\end{equation}
Whether the unrestricted constant equals either restricted constant
remains open.

\subsection{Dual exponents and nonlinear duality}
\label{subsec:general-p-duality}

The pair of dual equations written by Frank and Loss at $p=3/2$
\cite[Section~9, pp.~24--25]{FrankLossZeroMode2022} is the self-dual
case of a correspondence between different critical exponents.
Inverting a nonlinear flux is also the mechanism of nonlinear Hodge
duality \cite{SibnerSibner1970,MariniOtway2014}.
Here we identify the paired critical exponents and the interchange of
the cylindrical classes.  For $1<p<3$, put
\begin{equation}\label{eq:dual-exponent}
 \widehat p:=(p^*)'=\frac{3p}{4p-3}.
\end{equation}
Then $1<\widehat p<3$ and
\[
 \widehat p^*=p',\qquad
 \widehat{\widehat p}=p,\qquad
 (\widehat p-1)(p^*-1)=1.
\]
In particular, $\widehat p=p$ if and only if $p=3/2$.
We denote the quotient in \eqref{eq:global-constant} by
\[
 Q_p(u):=\frac{\|\curl u\|_p^p}{\|u+w(u)\|_{p^*}^p},
 \qquad \curl u\ne0,
\]
where the gauge space and the minimizing correction $w(u)$ are
understood at the exponent of the corresponding problem.
Thus $Q_{3/2}=Q$.  Define the nonlinear map
\[
 \mathfrak D_p(u):=|\curl u|^{p-2}\curl u,
\]
where the expression is defined to be zero when $\curl u=0$.
For an arbitrary finite-energy field, $\mathfrak D_p(u)$ is merely an
element of $L^{p'}$; that it lies in the dual energy space is a
consequence of the equation.

\begin{theorem}[Duality of finite-energy solutions]
\label{thm:general-p-duality}
Let $1<p<3$, let $\widehat p$ be given by \eqref{eq:dual-exponent}, and let
$u\in\mathcal D^{1,p}(\curlname;\R^3)$ be a nonzero weak solution of
\eqref{eq:p-curl-curl}.  Then
$v:=\mathfrak D_p(u)\in\mathcal D^{1,\widehat p}(\curlname;\R^3)$ solves
\begin{equation}\label{eq:dual-p-PDE}
 \curl\bigl(|\curl v|^{\widehat p-2}\curl v\bigr)
 =|v|^{\widehat p^*-2}v
 \qquad\text{in }\R^3,
\end{equation}
and $\mathfrak D_{\widehat p}(v)=u$ almost everywhere.  Thus
$\mathfrak D_p$ is a bijection between the nonzero finite-energy solution
sets at $p$ and $\widehat p$, with inverse $\mathfrak D_{\widehat p}$.
It interchanges the cylindrical classes:
\[
 u\in\Tclass\ \Longrightarrow\ v\in\Sclass,
 \qquad
 u\in\Sclass\ \Longrightarrow\ v\in\Tclass,
\]
where each class is taken in the energy space of its respective exponent.
The two solutions satisfy
\begin{equation}\label{eq:dual-energy-mass}
 \int_{\R^3}|\curl u|^p\dd x
 =\int_{\R^3}|u|^{p^*}\dd x
 =\int_{\R^3}|v|^{\widehat p^*}\dd x
 =\int_{\R^3}|\curl v|^{\widehat p}\dd x=:m,
\end{equation}
and consequently
\begin{equation}\label{eq:dual-energy-quotients}
 J_p(u)=J_{\widehat p}(v)=\frac m3,
 \qquad Q_p(u)^{1/p}=Q_{\widehat p}(v)^{1/\widehat p}=m^{1/3}.
\end{equation}
In particular, the correspondence preserves ground states, both in the
full space and between the paired cylindrical classes.
\end{theorem}

\begin{proof}
By definition, $v\in L^{p'}=L^{\widehat p^*}$, and the equation for $u$ gives
\[
 \curl v=|u|^{p^*-2}u\in L^{\widehat p},
\]
because $(p^*-1)\widehat p=p^*$.  The identity
$(\widehat p-1)(p^*-1)=1$ now yields
\[
 |\curl v|^{\widehat p-2}\curl v=u.
\]
On the other hand, $|v|^{p'-2}v=\curl u$.  Taking the distributional curl
of the preceding identity proves \eqref{eq:dual-p-PDE}.  The same argument
with $p$ and $\widehat p$ interchanged gives the inverse correspondence.
The cylindrical curl formulas interchange toroidal and poloidal fields,
while multiplication by a scalar function of their length preserves each
type.  These statements remain valid distributionally, by rotational
averaging and smooth approximation in the curl graph norm.

Testing the equation for $u$ with $u$ gives the first equality in
\eqref{eq:dual-energy-mass}.  This testing is legitimate by density of
$\mathcal C_c^\infty(\R^3,\R^3)$ in the curl graph space: in a cutoff at
radius $R$, the extra curl term is bounded in $L^p$ by
$C\|u\|_{L^{p^*}(B_{2R}\setminus B_R)}\to0$ as $R\to\infty$;
mollification then gives smooth test fields.  The other equalities follow directly from the two pointwise
power identities above.  Taking divergence of both equations shows that
$u$ and $v$ satisfy their respective nonlinear gauge conditions.  By the
annihilator identity, both are therefore the norm-minimizing
representatives of their gauge classes.  Since
$1/p-1/p^*=1/\widehat p-1/\widehat p^*=1/3$, formula
\eqref{eq:dual-energy-quotients} follows.  Finally, the bijection of solution
sets preserves their energies, and hence also their least energies.
\end{proof}

Together with the attainment results of
\cite{MederskiSzulkin2025}, the solution correspondence gives the
following identities for the sharp constants.

\begin{theorem}[Duality of the sharp constants]
\label{thm:general-p-constant-duality}
For every $1<p<3$ one has
\begin{align}
 S_{p,\curlname}^{1/p}
 &=S_{\widehat p,\curlname}^{1/\widehat p},
 \label{eq:global-p-constant-duality}\\
 \bigl(S_{p,\curlname}^{\Tclass}\bigr)^{1/p}
 &=\bigl(S_{\widehat p,\curlname}^{\Sclass}\bigr)^{1/\widehat p},
 \qquad
 \bigl(S_{p,\curlname}^{\Sclass}\bigr)^{1/p}
 =\bigl(S_{\widehat p,\curlname}^{\Tclass}\bigr)^{1/\widehat p}.
 \label{eq:cylindrical-p-constant-duality}
\end{align}
If a minimizer is rescaled to solve \eqref{eq:p-curl-curl}, its image
under $\mathfrak D_p$ attains the corresponding dual constant and solves
the dual equation.  Thus attainment transfers in both directions.
\end{theorem}

\begin{proof}
By \cite[Theorems~1.1 and~1.4]{MederskiSzulkin2025}, the unrestricted
and the two cylindrical constants are attained, and their gauge-minimal
minimizers can be rescaled to solve \eqref{eq:p-curl-curl}.
Let $u$ be such a minimizer for $S_{p,\curlname}$.
Theorem~\ref{thm:general-p-duality} gives
\[
 S_{\widehat p,\curlname}^{1/\widehat p}
 \le Q_{\widehat p}(\mathfrak D_pu)^{1/\widehat p}
 =Q_p(u)^{1/p}=S_{p,\curlname}^{1/p}.
\]
Interchanging $p$ and $\widehat p$ proves
\eqref{eq:global-p-constant-duality}.
Starting instead with a toroidal minimizer yields
\[
 \bigl(S_{\widehat p,\curlname}^{\Sclass}\bigr)^{1/\widehat p}
 \le\bigl(S_{p,\curlname}^{\Tclass}\bigr)^{1/p},
\]
because the dual solution is poloidal.  Applying the inverse
correspondence to a poloidal minimizer gives the reverse inequality.
Interchanging the cylindrical classes proves the remaining identity.
Equality in these comparisons also shows that the image of every
normalized minimizer is a minimizer of the paired problem.
\end{proof}

\begin{corollary}[The exponent dual to two]
\label{cor:p-two-dual}
Let $S_{\mathrm{Sob},2}=3(\pi/2)^{4/3}$ be the classical scalar Sobolev constant in
dimension three.  Then
\[
 S_{6/5,\curlname}=S_{2,\curlname}^{3/5}
 >S_{\mathrm{Sob},2}^{3/5}=3^{3/5}\left(\frac\pi2\right)^{4/5}.
\]
The sharp inequality
\begin{equation}\label{eq:p-two-dual-inequality}
 \|v\|_2\le S_{2,\curlname}^{-1/2}\|\curl v\|_{6/5}
\end{equation}
holds for every $v\in L^2(\R^3,\R^3)$ satisfying $\diver v=0$ and
$\curl v\in L^{6/5}(\R^3,\R^3)$, and is attained.
If $u$ is any ground state furnished by
\cite[Theorem~1.2]{MederskiSzulkin2021}, then $v=\curl u$ is a ground state
of
\[
 \curl\bigl(|\curl v|^{-4/5}\curl v\bigr)=v
 \qquad\text{in }\R^3,
\]
and the inverse formula is $u=|\curl v|^{-4/5}\curl v$.  Their common
ground-state energy is
\[
 J_{6/5}(v)=J_2(u)
 =\frac13S_{2,\curlname}^{3/2}
 =\frac13S_{6/5,\curlname}^{5/2}.
\]
For the cylindrical constants one has
\[
 S_{6/5,\curlname}^{\Sclass}
 =\bigl(S_{2,\curlname}^{\Tclass}\bigr)^{3/5},
 \qquad
 S_{6/5,\curlname}^{\Tclass}
 =\bigl(S_{2,\curlname}^{\Sclass}\bigr)^{3/5}.
\]
If a toroidal $p=2$ minimizer solves \eqref{eq:p-curl-curl}, its curl is
a poloidal $p=6/5$ minimizer; the analogous assertion holds with the two
classes interchanged.
\end{corollary}

\begin{proof}
For $p=2$ one has $p^*=6$, $\widehat p=6/5$, and
$\widehat p^*=2$.  Theorem~\ref{thm:general-p-constant-duality}, that is,
\eqref{eq:global-p-constant-duality} and
\eqref{eq:cylindrical-p-constant-duality}, gives the constant identities.
The strict inequality $S_{2,\curlname}>S_{\mathrm{Sob},2}$ and
attainment at $p=2$ are precisely
\cite[Theorem~1.2(a)--(b)]{MederskiSzulkin2021}.
In $L^2$ the gauge-minimizing representative is the divergence-free one,
by the orthogonal Helmholtz decomposition.  Taking the $5/6$ power of
the $p=6/5$ Sobolev--curl inequality therefore gives
\eqref{eq:p-two-dual-inequality}.  Theorem~\ref{thm:general-p-duality}
transfers the $p=2$ ground state, proves attainment of this inequality,
and gives both the equation and the inverse formula.  The cylindrical
attainment statements use \cite[Theorem~1.4]{MederskiSzulkin2025} at $p=2$.
\end{proof}

The existence of ground states at $p=6/5$ is contained in the general
theorem of \cite{MederskiSzulkin2025}; what the corollary adds is their
explicit reconstruction from the $p=2$ theory of
\cite{MederskiSzulkin2021}, the exact relation between the sharp constants,
and the interchange of the two cylindrical classes.  For $p\ne3/2$ the
identities \eqref{eq:cylindrical-p-constant-duality} relate different
exponents and say nothing about which of $S_{p,\curlname}^{\Tclass}$ and
$S_{p,\curlname}^{\Sclass}$ is smaller.  Note also that $\mathfrak D_p$
does not preserve the quotient of arbitrary fields;
\eqref{eq:dual-energy-quotients} relies on $u$ being a solution.

\subsection{The self-dual exponent}

From now on we take $p=3/2$, so that $p^*=3$.  The equation becomes
\begin{equation}\label{eq:selfdual-PDE}
 \curl\bigl(\abs{\curl u}^{-1/2}\curl u\bigr)=\abs{u}u.
\end{equation}
For a field satisfying the Nehari identity
\[
 \int_{\R^3}\abs{\curl u}^{3/2}\dd x
 =\int_{\R^3}\abs{u}^{3}\dd x,
\]
we have
\[
 J_{3/2}(u)=\frac13\int_{\R^3}\abs{u}^{3}\dd x
 =\frac13\int_{\R^3}\abs{\curl u}^{3/2}\dd x.
\]

At this exponent the gauge-invariant quotient is $Q$ of
\eqref{eq:field-quotient}, the infimum in its denominator being taken
over $\Wspace$.  If $u\in\Mman$, then $w(u)=0$ and this denominator
is simply $\bigl(\int_{\R^3}\abs{u}^{3}\dd x\bigr)^{1/2}$.

We recall the following elementary gauge property, already used in
\cite[Section~1]{BieganowskiMederski2026}.  For the reader's convenience,
we include a short proof.

\begin{lemma}
\label{lem:T-gauge}
If $A=a(\rho,z)e_\theta\in L^3(\R^3,\R^3)$, then
\[
 \diver(\abs{A}A)=0
 \quad\text{in }\mathscr D'(\R^3).
\]
Consequently, the infimum of $\|A+w\|_3$ over curl-free $w\in L^3$
is attained at $w=0$.  In particular, for a finite-energy field $A$,
one has $w(A)=0$ in \eqref{eq:gauge-minimizer}.
\end{lemma}

\begin{proof}
For $\varphi\in\mathcal C_c^\infty(\R^3)$, cylindrical coordinates give
\[
 \int_{\R^3}|A|A\cdot\nabla\varphi\dd x
 =
 \int_{\R}\int_0^\infty |a|a
 \left(\int_0^{2\pi}
   \partial_\theta\varphi(\rho\cos\theta,\rho\sin\theta,z)
   \dd\theta\right)\dd\rho\,\dd z=0.
\]
The integral is absolutely convergent because $|A|A\in L^{3/2}$ and
$\nabla\varphi$ is bounded and compactly supported.  This argument
also rules out a distribution supported on the axis.
The conclusion follows from
\eqref{eq:gauge-EL}--\eqref{eq:nonlinear-gauge} and uniqueness of the
gauge minimizer.
\end{proof}

\section{The hyperbolic problem}\label{sec:hyperbolic}

In this section we fix the hyperbolic notation and collect the properties
of the reduced variational problem.  For the hyperbolic geometry used below
we refer to \cite{Ratcliffe2019}.

\subsection{Upper half-plane model}

We use the upper half-plane model
\[
 \Hh^2:=\bigl\{(\rho,z)\in\R^2:\rho>0\bigr\},
 \qquad
 g:=\frac{\dd\rho^2+\dd z^2}{\rho^2}.
\]
It is the complete, simply connected two-dimensional Riemannian manifold of
constant sectional curvature $-1$.  Its volume element is
\[
 \dd V_g=\frac{\dd\rho\,\dd z}{\rho^2}.
\]
For a scalar function $\psi$,
\[
 \abs{\nabla_g\psi}_g=\rho\abs{\nabla\psi},
\]
where $\nabla$ on the right is the ordinary Euclidean gradient in the
$(\rho,z)$ variables.

The hyperbolic $q$-Laplacian is
\begin{equation}\label{eq:H-p-laplace}
 \Delta_{q,g}\psi
 :=\diver_g\bigl(\abs{\nabla_g\psi}_g^{q-2}\nabla_g\psi\bigr).
\end{equation}
In dimension two, formula \eqref{eq:H-p-laplace} becomes
\[
 \Delta_{q,g}\psi
 =\rho^2\diver_{\rho,z}
 \left(\rho^{q-2}\abs{\nabla\psi}^{q-2}\nabla\psi\right).
\]
In particular,
\begin{equation}\label{eq:H-threehalf}
 \Delta_{3/2,g}\psi
 =\rho^2\diver_{\rho,z}
 \left(\rho^{-1/2}\abs{\nabla\psi}^{-1/2}\nabla\psi\right).
\end{equation}

\subsection{Hyperbolic polar coordinates}

Fix the point
\[
 O:=(1,0)\in\Hh^2
\]
and let
\[
 t:=\dist_g(O,(\rho,z)).
\]
The distance formula in the upper half-plane gives
\begin{equation}\label{eq:distance-formula}
 \cosh t
 =1+\frac{(\rho-1)^2+z^2}{2\rho}
 =\frac{1+\rho^2+z^2}{2\rho}.
\end{equation}
Throughout we use the abbreviations
\[
 s:=1+\rho^2+z^2,
 \qquad
 \eta:=\frac{2\rho}{s},
\]
so that
\[
 \eta=\sech t,\qquad 0<\eta\le1.
\]
The value $\eta=1$ corresponds to the centre $(\rho,z)=(1,0)$, and
$\eta\to0$ to hyperbolic infinity.

In geodesic polar coordinates $(t,\varphi)$ centred at $O$,
\[
 \dd V_g=\sinh t\,\dd t\,\dd\varphi.
\]
For a radial function $\psi=\psi(t)$,
\[
 \Delta_{q,g}\psi
 =\frac1{\sinh t}\frac{\dd}{\dd t}
 \left(\sinh t\,\abs{\psi'}^{q-2}\psi'\right).
\]
For $a>0$, let $\psi_a(t)$ be the radial solution about $O$ of
\begin{equation}\label{eq:hyperbolic-radial-ODE}
 -\frac1{\sinh t}\frac{\dd}{\dd t}
 \left(\sinh t\,\abs{\psi_a'}^{-1/2}\psi_a'\right)
 =\psi_a^2,
 \qquad t>0,
\end{equation}
with $\psi_a(0)=a$ and $\psi_a'(0)=0$, considered up to its first zero.

A \emph{hyperbolic isometry} is a bijection $F:\Hh^2\to\Hh^2$
preserving the distance $\dist_g$.  For $\rho_0>0$ and $z_0\in\R$,
the map
\[
 F_{\rho_0,z_0}(\rho,z)=(\rho_0\rho,\rho_0z+z_0)
\]
is an isometry sending $O$ to $(\rho_0,z_0)$; these maps act
transitively on $\Hh^2$, preserve the two integrals defining
$\Lambda_{\Hh^2}$ in \eqref{eq:Lambda-H} and equation
\eqref{eq:hyperbolic-PDE}, and correspond to the Euclidean maps
$T_{1/\rho_0,z_0}$ of \eqref{eq:euclidean-moves} under
$\psi\mapsto A(\psi)$.

\subsection{The reduced variational problem}

The hyperbolic Poincar\'e inequality identifies the completion
$\mathcal D^{1,3/2}(\Hh^2)$ of Section~\ref{sec:introduction} with
$W^{1,3/2}(\Hh^2)$, with equivalent norms.  We study the constant
$\Lambda_{\Hh^2}$ defined in \eqref{eq:Lambda-H}; since the critical
Sobolev exponent of the $3/2$-Laplacian in dimension two is
$(3/2)^*_2=6$, the exponent $3$ there is subcritical.

The positivity of $\Lambda_{\Hh^2}$ follows by interpolation.  Let
$S_{\mathrm E}=2(4\pi^2/(9\sqrt3))^{3/4}$ be the sharp constant of the
Euclidean Sobolev inequality in dimension two with gradient exponent
$3/2$ \cite{Aubin1976,Talenti1976}.  The hyperbolic Poincar\'e
inequality $\norm{\psi}{3/2}\le\frac32\norm{\nabla_g\psi}{3/2}$ and the
critical embedding
$\norm{\psi}{6}\le S_{\mathrm E}^{-2/3}\norm{\nabla_g\psi}{3/2}$ of
\cite[(6.2)]{DuttaSandeep2025}, all norms taken with respect to
$\dd V_g$, give
\begin{equation}\label{eq:H-continuous-L3-embedding}
 \norm{\psi}{3}
 \le\norm{\psi}{3/2}^{1/3}\norm{\psi}{6}^{2/3}
 \le(3/2)^{1/3}S_{\mathrm E}^{-4/9}
                           \norm{\nabla_g\psi}{3/2},
\end{equation}
hence $\Lambda_{\Hh^2}\ge\sqrt{2/3}\,S_{\mathrm E}^{2/3}>0$.

The following compactness statement is the case $n=2$, $p=3/2$,
$q=3$ of \cite[Lemma~6.1]{DuttaSandeep2025}.

\begin{lemma}[Radial compactness at a fixed centre]
\label{lem:fixed-centre-radial-compactness}
Let $\mathcal D_{\mathrm{rad},O}^{1,3/2}(\Hh^2)$ be the closure, in the
homogeneous energy norm, of the smooth compactly supported functions that are
radial about $O$, where
\[
 \norm{f}{\mathcal D^{1,3/2}}
 :=\left(\int_{\Hh^2}\abs{\nabla_gf}_g^{3/2}\dd V_g\right)^{2/3}.
\]
Then the following embedding is compact:
\[
 \mathcal D_{\mathrm{rad},O}^{1,3/2}(\Hh^2)
 \hookrightarrow L^3(\Hh^2).
\]
\end{lemma}

For equation \eqref{eq:hyperbolic-PDE}, the symmetry theorem of
Dutta and Sandeep \cite[Theorem~1.1]{DuttaSandeep2025} states that
every positive finite-energy solution is radial about some point of
$\Hh^2$, strictly decreasing in the hyperbolic distance $t$, and
satisfies $\psi(t)\asymp|\psi'(t)|\asymp\e^{-2t}$ as $t\to\infty$.
An extremal for \eqref{eq:Lambda-H} exists by
\cite[Theorem~6.2]{DuttaSandeep2025}.

Every real minimizer has one sign: otherwise, applying
\eqref{eq:Lambda-H} to its positive and negative parts and adding their
gradient energies contradicts minimality, since
$\sqrt{m_+}+\sqrt{m_-}>\sqrt{m_++m_-}$, where
$m_\pm:=\int_{\Hh^2}|\psi_\pm|^3\dd V_g>0$.
After changing the sign if necessary, local
regularity and the strong maximum principle \cite{Vazquez1984} give
strict positivity.  After a positive rescaling we may therefore choose
a positive extremal $\psi_*$ solving \eqref{eq:hyperbolic-PDE}.  By applying a hyperbolic
isometry, we place its centre at $O=(1,0)$ and write $\psi_*=\psi_*(t)$.
It satisfies \eqref{eq:hyperbolic-radial-ODE} with
\begin{equation}\label{eq:hyperbolic-boundary}
 \psi_*'(0)=0,
 \qquad
 \psi_*(t)>0,
 \qquad
 \psi_*'(t)<0\ (t>0),
 \qquad
 \psi_*(t)\to0\quad(t\to\infty).
\end{equation}
We do not know a closed formula for $\psi_*$.

\subsection{Uniqueness and the shooting trichotomy}
\label{subsec:profile-uniqueness}

The uniqueness of the positive finite-energy profile is the following
classification of the initial-value solutions $\psi_a$ of
\eqref{eq:hyperbolic-radial-ODE}.

\begin{theorem}
\label{thm:hyperbolic-profile-uniqueness}
There exists a unique number $a_*>0$ with the following properties:
\begin{enumerate}[label=\textup{(\roman*)},leftmargin=9mm]
 \item if $0<a<a_*$, then $\psi_a$ remains positive, satisfies
\[
  \psi_a(t)\to0,
  \qquad
  \e^{2t}\psi_a(t)\to\infty
  \qquad(t\to\infty),
 \]
 and does not belong to $\mathcal D^{1,3/2}(\Hh^2)$;
 \item if $a=a_*$, then $\psi_{a_*}$ remains positive and, for some
 $C_*>0$,
 \begin{equation}\label{eq:rapid-hyperbolic-decay}
  \e^{2t}\psi_{a_*}(t)\to C_*
  \qquad(t\to\infty);
 \end{equation}
 this is the unique positive radial solution of
 \eqref{eq:hyperbolic-PDE} in $\mathcal D^{1,3/2}(\Hh^2)$ with centre $O$;
 \item if $a>a_*$, then $\psi_a$ has a zero at a finite hyperbolic radius.
\end{enumerate}
Consequently every positive finite-energy solution of
\eqref{eq:hyperbolic-PDE} is obtained from $\psi_{a_*}$ by a hyperbolic
isometry, and $\psi_*=\psi_{a_*}$.
\end{theorem}

The finite-energy assumption is essential: the boundary conditions
\eqref{eq:hyperbolic-boundary} alone are satisfied by a continuum of
positive radial solutions, which tend to zero too slowly to have finite
energy.  Dutta and Sandeep \cite[Theorem~1.1]{DuttaSandeep2025} prove
symmetry and two-sided decay but not uniqueness of the nonlinear
profile; the substitution $r=(\e^{2t}-1)/2$ turns the radial equation
into a weighted Euclidean problem covered by Kawano, Yanagida and
Yotsutani \cite[Theorem~1]{KawanoYanagidaYotsutani1993}, whose
hypotheses we verify below.

\begin{proof}[Proof of Theorem~\ref{thm:hyperbolic-profile-uniqueness}]
We transform the radial equation into a weighted Euclidean equation of the
type studied by Kawano, Yanagida and Yotsutani
\cite{KawanoYanagidaYotsutani1993}.  Set
\[
 r=\frac{1}{\coth t-1}=\frac{\e^{2t}-1}{2},
 \qquad
 v(r)=\psi(t).
\]
Then
\[
 \frac{\dd r}{\dd t}=1+2r,
 \qquad
 \sinh t=\frac{r}{\sqrt{1+2r}},
\]
and \eqref{eq:hyperbolic-radial-ODE} becomes
\begin{equation}\label{eq:weighted-Euclidean-profile}
 -\frac{\dd}{\dd r}
 \left(r\abs{v'}^{-1/2}v'\right)
 =rK(r)v^2,
 \qquad
 K(r):=(1+2r)^{-3/2}.
\end{equation}
The two radial integrals transform as
\begin{align*}
 \int_0^\infty\abs{\psi'}^{3/2}\sinh t\dd t
 &=\int_0^\infty r\abs{v'}^{3/2}\dd r,
 \\
 \int_0^\infty\abs{\psi}^{3}\sinh t\dd t
 &=\int_0^\infty rK(r)\abs{v}^{3}\dd r.
\end{align*}

Equation \eqref{eq:weighted-Euclidean-profile} is the radial equation in
$\R^n$ considered in \cite{KawanoYanagidaYotsutani1993}, with
\[
 n=2,
 \qquad
 m=\frac32,
 \qquad
 q_{\mathrm{KYY}}=2.
\]
(In \cite{KawanoYanagidaYotsutani1993} the initial-value problem is continued
with $(v^+)^{q_{\mathrm{KYY}}}$, which coincides with
\eqref{eq:weighted-Euclidean-profile} up to the first zero of $v$, and this
is all that the crossing classification requires.)  We have $n>m>1$ and
$q_{\mathrm{KYY}}>m-1$; the weight $K$ is positive and smooth, with
$K(r)=1+O(r)$ at the origin, so the condition (K) of that paper holds with
exponent $\nu=0>-m$, and the additional requirement imposed there for $m>2$
is void.

For these parameters the Pohozaev-type function of
\cite{KawanoYanagidaYotsutani1993} is
\[
 \mathscr G(r)
 :=\frac13r^2K(r)-\frac13\int_0^r sK(s)\dd s.
\]
A direct calculation gives
\[
 \mathscr G'(r)
 =\frac13r\bigl(rK(r)\bigr)'
 =\frac{r(1-r)}{3(1+2r)^{5/2}}.
\]
Hence $\mathscr G'>0$ on $(0,1)$ and $\mathscr G'<0$ on $(1,\infty)$.
Since $\mathscr G(0)=0$, we get $\mathscr G\ge0$ on $(0,1)$,
$\mathscr G'\le0$ on $(1,\infty)$ and $\mathscr G\not\equiv0$, which is
hypothesis~(G) of \cite[Theorem~1]{KawanoYanagidaYotsutani1993} with
$R_1=1$.  To determine which alternative in that theorem applies,
we use the positive finite-energy solution provided by
\cite[Theorem~6.2]{DuttaSandeep2025}.  Its symmetry and sharp decay
\cite[Theorem~1.1]{DuttaSandeep2025} allow us to place its centre at $O$
and give
\[
 0<\liminf_{r\to\infty}rv(r)
 \le\limsup_{r\to\infty}rv(r)<\infty.
\]
This solution is neither crossing nor slowly decaying, so the mixed
case of the cited classification applies.  Since
$(n-m)/(m-1)=1$, there is exactly one $a_*>0$ for which
\begin{equation}\label{eq:KYY-rapid-limit}
 rv(r;a_*)\to L_*\in(0,\infty)
 \qquad(r\to\infty),
\end{equation}
whereas $rv(r;a)\to\infty$ as $r\to\infty$ for $0<a<a_*$, and the
solutions with $a>a_*$ cross zero.
Since $r\sim\frac12\e^{2t}$ as $t\to\infty$, this gives the
claimed decay alternatives, with $C_*=2L_*$, once we verify that
every global positive solution tends to zero.

Every global positive $\psi_a$ is decreasing and therefore has a limit
$\ell\ge0$.  Put $z=\sqrt{-\psi_a'}$.  The radial equation gives
\begin{equation}\label{eq:z-slow-limit}
 z(t)=\frac{1}{\sinh t}
 \int_0^t\psi_a(s)^2\sinh s\dd s.
\end{equation}
If $\ell>0$, l'H\^opital's rule in \eqref{eq:z-slow-limit} yields
$z(t)\to\ell^2$, and hence $\psi_a'(t)\to-\ell^4$, a contradiction.  Thus
$\ell=0$.

Every positive finite-energy solution is radial and has the sharp decay
used above, so, after moving its centre to $O$, it must correspond to
the unique rapid branch.  The previously constructed finite-energy
solution belongs to this branch; hence the rapid solution itself has
finite energy, and it coincides with $\psi_*$.  In particular, no
slowly decaying solution $\psi_a$ with $a<a_*$ has finite energy,
which is the last assertion in \textup{(i)}.
\end{proof}

There are earlier hyperbolic versions for the Laplace--Beltrami
operator.  Mancini and Sandeep \cite[Theorems~1.3--1.4]{ManciniSandeep2008}
prove uniqueness, up to hyperbolic isometries, of positive
finite-energy solutions of the semilinear problem, including
$-\Delta_gu=u^2$ on $\Hh^2$; for $\Hh^n$, $n\ge3$, Bonforte, Gazzola,
Grillo and V\'azquez \cite[Theorem~A]{BonforteEtAl2013} describe the
analogous slow-decay, ground-state and zero-crossing alternatives.
The quasilinear classification of Muratori and Soave
\cite[Theorems~1.3--1.5]{MuratoriSoave2023} concerns critical and
supercritical source powers $q\ge p^*-1$, which excludes the present
source power $2$.

\begin{remark}[Fixing the coefficient in the equation]
For $\lambda>0$, the rescaling $\widetilde\psi=\lambda^{2/3}\psi$
transforms $-\Delta_{3/2,g}\psi=\lambda\psi^2$ into the coefficient-one
equation.  Uniqueness of positive extremals of the quotient is therefore
understood up to a positive factor and hyperbolic isometries.
\end{remark}

\begin{remark}[Regularity at the hyperbolic centre]
\label{rem:profile-regularity}
Let $O$ be the centre of $\psi_*$ and $a_*=\psi_*(O)$.
We use the same symbol for the radial profile and the corresponding
function on $\Hh^2$.
By \cite[Theorem~1.1]{DuttaSandeep2025}, the profile is $\mathcal C^1$
at $t=0$, satisfies $\psi_*'(0)=0$, and is smooth for $t>0$.
In the present case its precise spatial regularity is
\[
 \psi_*\in \mathcal C_{\mathrm{loc}}^{2,1}(\Hh^2)
             \cap \mathcal C^\infty(\Hh^2\setminus\{O\}),
 \qquad
 \psi_*\notin \mathcal C^3(\Hh^2).
\]
Here $\mathcal C^{2,1}_{\mathrm{loc}}$ means that the second derivatives
are locally Lipschitz in smooth coordinates.

Indeed, integrating the radial equation from the centre gives
\[
 \psi_*'(t)
 =-\left(\frac{\int_0^t\psi_*(r)^2\sinh r\dd r}{\sinh t}\right)^2
 =-t^2\left(
       \int_0^1\psi_*(tr)^2\frac{\sinh(tr)}{\sinh t}\dd r
       \right)^2.
\]
The quotient in the last integral extends smoothly to $t=0$,
with value $r$.  Thus the integral preserves $\mathcal C^k$ regularity
of the profile, and the identity bootstraps $\psi_*$ to
$\mathcal C^\infty([0,\infty))$, with one-sided derivatives at zero.
The integral at zero equals $a_*^2/2$.  Consequently, as $t\to0^+$,
\[
 \psi_*'(0)=\psi_*''(0)=0,\qquad
 \psi_*'''(0+)=-\frac{a_*^4}{2},\qquad
 \psi_*(t)=a_*-\frac{a_*^4}{12}t^3+O(t^4).
\]
In geodesic normal coordinates $y$ centred at $O$, we have $t=|y|$.
The radial Hessian formula, together with
$\psi_*'(t)=O(t^2)$, $\psi_*''(t)=O(t)$ and boundedness of
$\psi_*'''$ near zero, shows that the spatial second derivatives
are locally Lipschitz.
However, the restriction to a geodesic through $O$, parametrized
by signed arclength $s$, is $\psi_*(|s|)$.
Its third derivative has limits $-a_*^4/2$ as $s\to0^+$
and $a_*^4/2$ as $s\to0^-$.
Thus the spatial function is not $\mathcal C^3$ at its centre.
\end{remark}

\section{Reduction of the toroidal quotient}\label{sec:reduction}

We now return to vector fields.  A toroidal field $a(\rho,z)e_\theta$
is determined by its stream function $\psi=\rho a$, and the two integrals
entering the quotient $Q$ are hyperbolic integrals of $\psi$.  We make
this correspondence precise on the level of the energy spaces and of the
Euler--Lagrange equations.

Hyperbolic reductions already occur in gauge theory: Witten's
$SO(3)$-invariant self-dual Yang--Mills fields in $\R^4$ reduce to an
abelian Higgs model on $\Hh^2$ \cite{Witten1977}.  This concerns a
different equation and symmetry reduction from the present toroidal
problem in $\R^3$.  For our ansatz the identities are exact for every
$1<p<3$, with $p^*=3p/(3-p)$:
\[
 \int_{\R^3}|\curl A(\psi)|^p\dd x
 =2\pi\int_{\Hh^2}\rho^{3-2p}|\nabla_g\psi|_g^p\dd V_g,
 \qquad
 \int_{\R^3}|A(\psi)|^{p^*}\dd x
 =2\pi\int_{\Hh^2}\rho^{3-p^*}|\psi|^{p^*}\dd V_g.
\]
What is specific to $p=3/2$ is that both reduced integrals are
unweighted hyperbolic integrals.

\begin{proposition}[The $\Tclass$--$\Hh^2$ correspondence]
\label{prop:T-H-correspondence}
For $\psi\in\mathcal D^{1,3/2}(\Hh^2)$ define
\begin{equation}\label{eq:A-from-psi}
 A(\rho,\theta,z):=\frac{\psi(\rho,z)}{\rho}\,e_\theta.
\end{equation}
This defines a linear bijection from $\mathcal D^{1,3/2}(\Hh^2)$ onto
$\Tclass\cap\Dcurlhalf$.  In the weak sense,
\begin{equation}\label{eq:curl-A-psi}
 \curl A
 =-\frac{\psi_z}{\rho}e_\rho
  +\frac{\psi_\rho}{\rho}e_3.
\end{equation}
Moreover,
\begin{align}
 \int_{\R^3}\abs{A}^{3}\dd x
 &=2\pi\int_{\Hh^2}\abs{\psi}^{3}\dd V_g,
 \label{eq:T-L3-H}\\
 \int_{\R^3}\abs{\curl A}^{3/2}\dd x
 &=2\pi\int_{\Hh^2}\abs{\nabla_g\psi}_g^{3/2}\dd V_g.
 \label{eq:T-curl-H}
\end{align}
Consequently,
\begin{equation}\label{eq:T-constant-H}
 S_{3/2,\curlname}^{\Tclass}=\sqrt{2\pi}\,\Lambda_{\Hh^2}.
\end{equation}
\end{proposition}

\begin{proof}
For a toroidal field $A=a(\rho,z)e_\theta$,
\[
 \curl A=-a_z e_\rho+\left(a_\rho+\frac a\rho\right)e_3.
\]
Taking $a=\psi/\rho$ gives \eqref{eq:curl-A-psi}.  Since
$\dd x=\rho\,\dd\rho\,\dd\theta\,\dd z$,
\[
 \int_{\R^3}\abs{A}^3\dd x
 =2\pi\int_0^\infty\int_\R
 \frac{\abs{\psi}^3}{\rho^3}\rho\,\dd z\,\dd\rho
 =2\pi\int_{\Hh^2}\abs{\psi}^3\dd V_g.
\]
Similarly,
\[
 \abs{\curl A}=\frac{\abs{\nabla\psi}}{\rho}
 =\frac{\abs{\nabla_g\psi}_g}{\rho^2},
\]
which yields \eqref{eq:T-curl-H}.  These computations apply first to $\psi\in\mathcal C_c^\infty(\Hh^2)$, for which
the field $A$ is smooth and compactly supported away from the symmetry
axis.  By the continuous embedding \eqref{eq:H-continuous-L3-embedding} and
the two norm identities, the map extends to the homogeneous completion, and
the distributional curl of the limit field carries no contribution
supported on the axis.

To prove that the map is onto, let $A\in\Tclass\cap\Dcurlhalf$.
Cutoff at spatial infinity, followed by convolution with a radial smooth
kernel, approximates $A$ in $L^3$ and its curl in $L^{3/2}$ by smooth,
compactly supported toroidal fields $A_j=a_j(\rho,z)e_\theta$.
The cutoff error is controlled by
$\norm{\nabla\chi_R\times A}{3/2}\le
\norm{\nabla\chi_R}{3}\norm{A}{L^3(\{R<\abs x<2R\})}$,
and radial convolution preserves the toroidal symmetry.
Smoothness at the axis gives $a_j=O(\rho)$, so
$\psi_j:=\rho a_j=O(\rho^2)$ and $\nabla\psi_j=O(\rho)$ locally.
Multiplying $\psi_j$ by a smooth cutoff which is zero for $\rho<\eps$ and
one for $\rho>2\eps$ makes it compactly supported in $\Hh^2$.
The lost hyperbolic $L^3$ mass is $O(\eps^5)$ and the additional gradient
energy is $O(\eps^2)$, using the weights $\rho^{-2}$ and $\rho^{-1/2}$
in \eqref{eq:T-L3-H}--\eqref{eq:T-curl-H}.  Hence each $\psi_j$ lies in
$\mathcal D^{1,3/2}(\Hh^2)$; the norm identities make $(\psi_j)$ Cauchy
there, and its limit maps to $A$.
By Lemma~\ref{lem:T-gauge}, the gauge correction is zero for every
admissible toroidal field.  Taking the infimum over this bijection proves
\eqref{eq:T-constant-H}.
\end{proof}

\begin{proposition}
\label{prop:T-PDE-H}
Let $\psi\in\mathcal D^{1,3/2}(\Hh^2)$ and let $A$ be given by
\eqref{eq:A-from-psi}.  On $\{\rho>0\}$, in the distributional sense,
\begin{equation}\label{eq:curl-nonlinear-A}
 \curl\bigl(\abs{\curl A}^{-1/2}\curl A\bigr)
 =-\frac1{\rho^2}\Delta_{3/2,g}\psi\,e_\theta.
\end{equation}
Hence $A$ solves \eqref{eq:selfdual-PDE} if and only if
\[
 -\Delta_{3/2,g}\psi=\abs{\psi}\psi
 \qquad\text{in }\Hh^2.
\]
\end{proposition}

\begin{proof}
Put
\[
 m:=\rho^{-1/2}\abs{\nabla\psi}^{-1/2}
 \quad\text{where }\nabla\psi\ne0,
 \qquad m:=0\quad\text{where }\nabla\psi=0.
\]
Using \eqref{eq:curl-A-psi},
\[
 \abs{\curl A}^{-1/2}\curl A
 =-m\psi_z e_\rho+m\psi_\rho e_3.
\]
The azimuthal component of its curl is
\[
 -\partial_z(m\psi_z)-\partial_\rho(m\psi_\rho)
 =-\diver_{\rho,z}
 \bigl(\rho^{-1/2}\abs{\nabla\psi}^{-1/2}\nabla\psi\bigr),
\]
which is \eqref{eq:curl-nonlinear-A} by \eqref{eq:H-threehalf}.  At zeros
of $\nabla\psi$ the nonlinear vector expression is assigned its continuous
value zero.  Approximation in the hyperbolic energy norm proves the
distributional identity on $\{\rho>0\}$ for every admissible $\psi$.
Since $\abs{A}A=\rho^{-2}\abs{\psi}\psi\,e_\theta$, the two weak equations
are equivalent away from the axis; angular averaging reduces arbitrary
vector tests there to their toroidal component.

If the scalar equation holds, put
$B=\abs{\curl A}^{-1/2}\curl A\in L^3$ and
$F=\abs A A\in L^{3/2}$.  To extend $\curl B=F$ across the axis, test
with $\chi_\eps(\rho)\Phi$, where $\Phi\in\mathcal C_c^\infty(\R^3,\R^3)$
and $\chi_\eps$ vanishes for $\rho<\eps$ and equals one for
$\rho>2\eps$.  On the fixed support of $\Phi$,
\[
 \norm{\nabla\chi_\eps}{L^{3/2}}=O(\eps^{1/3}).
\]
Hölder's inequality therefore makes the boundary term involving
$B\cdot(\nabla\chi_\eps\times\Phi)$ tend to zero.  The other terms
converge by local integrability.  Thus the vector equation holds on all of
$\R^3$, without an axial source.  Conversely, testing the global vector
equation by toroidal fields supported away from the axis gives the scalar
weak equation; density extends its test class to
$\mathcal D^{1,3/2}(\Hh^2)$.
\end{proof}

\section{Nonlinear self-duality and the poloidal minimizer}\label{sec:duality}

At $p=3/2$, Theorem~\ref{thm:general-p-duality} makes
$\mathfrak D(u)=|\curl u|^{-1/2}\curl u$ an involution on finite-energy
solutions and interchanges the two cylindrical classes.

Let $\psi_*$ be the positive hyperbolic minimizer centred at $O$ and
normalized by $-\Delta_{3/2,g}\psi_*=\psi_*^2$, and put
$A_*=\rho^{-1}\psi_*e_\theta$.
The correspondence and weak-equation identities in
Propositions~\ref{prop:T-H-correspondence} and~\ref{prop:T-PDE-H}
show that $A_*$ is a toroidal extremal solving
\eqref{eq:selfdual-PDE}.  Define
\begin{equation}\label{eq:U-star-dual-definition}
 U_*:=\mathfrak D(A_*)=|\curl A_*|^{-1/2}\curl A_*.
\end{equation}
Theorems~\ref{thm:general-p-duality}
and~\ref{thm:general-p-constant-duality} show directly that $U_*$
solves the same equation and is a poloidal extremal.
We use these facts here to describe its geometry and prove the
zero-mode exclusion.  Their role in the sharp-constant identities
is detailed in Section~\ref{sec:proof-main}.

\subsection{The semi-explicit Euclidean formula}

We now express $U_*=\mathfrak D(A_*)$ in terms of the radial hyperbolic
profile.  Define
\begin{equation}\label{eq:N-field}
 N(x):=
 \begin{pmatrix}
  2x_1z\\
  2x_2z\\
  1+z^2-\rho^2
 \end{pmatrix}
 =2\rho z\,e_\rho+(1+z^2-\rho^2)e_3.
\end{equation}

The elementary identities
\[
 \abs{N}^2
 =(1+z^2-\rho^2)^2+4\rho^2z^2
 =s^2-4\rho^2
 =s^2(1-\eta^2)
\]
and \eqref{eq:distance-formula} imply
\begin{equation}\label{eq:N-hyperbolic-norm}
 \abs{N}=2\rho\sinh t.
\end{equation}
For $\rho>0$ and $t>0$, differentiating
\eqref{eq:distance-formula} gives
\[
 t_\rho=-\frac{1+z^2-\rho^2}{2\rho^2\sinh t},
 \qquad
 t_z=\frac{z}{\rho\sinh t}.
\]
Consequently,
\[
 -t_z e_\rho+t_\rho e_3
 =-\frac{N}{2\rho^2\sinh t}.
\]

For $A_*=\psi_*(t)\rho^{-1}e_\theta$, formula
\eqref{eq:curl-A-psi} gives
\begin{equation}\label{eq:curl-A-star-N}
 \curl A_*
 =-\frac{\psi_*'(t)}{2\rho^3\sinh t}N.
\end{equation}
Since $\psi_*'<0$, \eqref{eq:N-hyperbolic-norm} yields
\begin{equation}\label{eq:curl-A-star-norm}
 \abs{\curl A_*}=\frac{-\psi_*'(t)}{\rho^2}.
\end{equation}
Inserting \eqref{eq:curl-A-star-N}--\eqref{eq:curl-A-star-norm} into
\eqref{eq:U-star-dual-definition} gives
\begin{equation}\label{eq:U-star-intermediate}
 U_*(x)=\frac{\sqrt{-\psi_*'(t)}}{2\rho^2\sinh t}\,N(x).
\end{equation}
Using $s=2\rho\cosh t=2\rho/\eta$ we arrive at the following formula.
The continuous endpoint extensions established below define the field
also on the axis and the circle $\{\rho=1,z=0\}$.

\begin{theorem}[Semi-explicit form of the $\Sclass$-extremal]
\label{thm:semi-explicit}
Let $\psi_*$ be the positive finite-energy solution of
\eqref{eq:hyperbolic-PDE} centred at $O$, as in
Theorem~\ref{thm:hyperbolic-profile-uniqueness}.  Define
\begin{equation}\label{eq:G-from-psi}
 G_*(\eta):=
 \frac{2\sqrt{-\psi_*'(t)}}
      {\eta\sqrt{1-\eta^2}},
 \qquad
 t=\operatorname{arcosh}\frac1\eta,
 \quad 0<\eta<1.
\end{equation}
We use the continuous extension of $G_*$ to $[0,1]$.
Then
\begin{equation}\label{eq:U-star-explicit}
  U_*(x)=
 \frac{G_*\!\left(\dfrac{2\rho}{1+\rho^2+z^2}\right)}
 {(1+\rho^2+z^2)^2}
 \begin{pmatrix}
  2x_1z\\[1mm]
  2x_2z\\[1mm]
  1+z^2-\rho^2
 \end{pmatrix}.
\end{equation}
The field $U_*$ has the following properties:
\begin{enumerate}[label=\textup{(\roman*)},leftmargin=9mm]
 \item $U_*\in\Sclass\cap\Mman$ and $w(U_*)=0$;
 \item $U_*$ solves the full equation \eqref{eq:selfdual-PDE};
 \item $U_*$ realizes $S_{3/2,\curlname}^{\Sclass}$;
 \item its coefficients in \eqref{eq:S-form} are
 \begin{align*}
  \alpha_*(\rho,z)
  &=\frac{2z}{s^2}G_*\!\left(\frac{2\rho}{s}\right),
  \\
  \beta_*(\rho,z)
  &=\frac{1+z^2-\rho^2}{s^2}
    G_*\!\left(\frac{2\rho}{s}\right),
 \end{align*}
 and therefore
 \[
  \alpha_*(\rho,-z)=-\alpha_*(\rho,z),
  \qquad
  \beta_*(\rho,-z)=\beta_*(\rho,z).
\]
\end{enumerate}
\end{theorem}

\begin{proof}
Formula \eqref{eq:U-star-explicit} is equivalent to
\eqref{eq:U-star-intermediate}.
Propositions~\ref{prop:T-H-correspondence} and~\ref{prop:T-PDE-H}
make $A_*$ a toroidal extremal solving \eqref{eq:selfdual-PDE}.
Theorems~\ref{thm:general-p-duality}
and~\ref{thm:general-p-constant-duality} then give assertions
\textup{(i)}--\textup{(iii)} for $U_*=\mathfrak D(A_*)$;
the nonlinear gauge condition is part of the solution duality.
Assertion~\textup{(iv)} follows from the decomposition of $N$
in \eqref{eq:N-field}.
\end{proof}

\begin{remark}[Endpoint values of the poloidal factor]
The factor $G_*$ extends continuously to the endpoints.
Substituting the centre expansion from Remark~\ref{rem:profile-regularity}
into \eqref{eq:G-from-psi}, with $a_*=\psi_*(0)$, gives
\[
 G_*(1)=a_*^2.
\]
At hyperbolic infinity the value of $G_*$ is determined by the flux
\[
 \mathcal J(t):=\sinh t\sqrt{-\psi_*'(t)}
   =\int_0^t\psi_*(s)^2\sinh s\dd s ,
\]
which increases to a finite limit $\mathcal J_\infty$ as $t\to\infty$
because $\psi_*(t)=O(\e^{-2t})$.  In the same limit,
$-\psi_*'(t)\sim4\mathcal J_\infty^2\e^{-2t}$ and
$\psi_*(t)\sim2\mathcal J_\infty^2\e^{-2t}$, and
\eqref{eq:G-from-psi} gives
\[
 G_*(0)=2\mathcal J_\infty=\sqrt{2C_*}=2\sqrt{L_*},
\]
where $C_*$ and $L_*$ are the limits in
\eqref{eq:rapid-hyperbolic-decay} and \eqref{eq:KYY-rapid-limit}.
\end{remark}

\subsection{Absence of zero modes}\label{subsec:no-zero}

We now prove Theorem~\ref{thm:extremals-no-zero}, adapting the angular
obstruction underlying \cite[Theorem~1.1]{Elton2018}.
Part~\textup{(a)} is an angular identity for bounded axisymmetric
poloidal potentials, which we isolate as a proposition; for
part~\textup{(b)} an explicit conformal map reduces the toroidal
extremal to that case, and we verify the $L^2$ integrability of the
transformed spinor, which does not follow from conformal covariance
alone.

The extremals are indeed not smooth.  Along $x=(1,0,z)$, as $z\to0$, one has
$t=2\operatorname{arsinh}(|z|/2)=|z|+O(|z|^3)$, so
Remark~\ref{rem:profile-regularity} gives
\[
 (A_*)_2(1,0,z)
 =a_*-\frac{a_*^4}{12}|z|^3+O(|z|^4).
\]
Thus $A_*$ is not $\mathcal C^3$ across $\{\rho=1,z=0\}$.
Moreover, $\curl U_*=|A_*|A_*=\rho^{-2}\psi_*(t)^2e_\theta$,
whose second Cartesian component on the same line is
$a_*^2-a_*^5|z|^3/6+O(|z|^4)$.
Hence $\curl U_*$ is not $\mathcal C^3$, so $U_*$ is not $\mathcal C^4$ there.
In particular, neither potential satisfies the smoothness assumption
in Elton's stated theorem.

\begin{proposition}[Bounded poloidal potentials have no zero modes]
\label{prop:poloidal-no-zero}
Let $B=b_\rho(\rho,z)e_\rho+b_z(\rho,z)e_3\in L^\infty(\R^3;\R^3)$.
Then $\ker_{L^2}\mathcal D_{cB}=\{0\}$ for every $c\in\R$, and
$\mathcal D_{cB}$ is self-adjoint on $H^1(\R^3;\mathbb C^2)$.
\end{proposition}

\begin{proof}
\emph{Step 1: angular decomposition for a bounded poloidal potential.}
Fix $c\in\R$.  The coefficients of $B$ are independent of $\theta$,
and there is no $e_\theta$ component.  If $\Psi\in L^2(\R^3;\mathbb C^2)$ solves
$\mathcal D_{cB}\Psi=0$ distributionally, then
$\mathcal D_0\Psi=c\boldsymbol\sigma\cdot B\Psi\in L^2$.
The Pauli identity
$|(\boldsymbol\sigma\cdot\xi)v|^2=|\xi|^2|v|^2$
and Fourier transformation give
\[
 \|\nabla\Psi\|_2=\|\mathcal D_0\Psi\|_2
 \le |c|\|B\|_\infty\|\Psi\|_2.
\]
Hence $\Psi\in H^1$.  Also, multiplication by
$c\boldsymbol\sigma\cdot B$ is bounded and Hermitian, so the
self-adjoint operator $\mathcal D_{cB}$ has domain $H^1$.

To exploit the symmetry, use the rotating spin frame
$S(\theta)=\exp(-i\theta\sigma_3/2)$.
Since $S(\theta+2\pi)=-S(\theta)$ whereas $\Psi$ is periodic,
$S(\theta)^*\Psi$ is antiperiodic.  Its Fourier frequencies are
therefore half-integers, and the angular modes have the form
\[
 \Psi_m=S(\theta)\e^{im\theta}u_m(\rho,z),
 \qquad m\in\mathbb Z+\tfrac12,
 \qquad u_m:(0,\infty)\times\R\to\mathbb C^2.
\]
The projections onto these modes preserve $H^1$ and commute with
$\mathcal D_{cB}$: they are averages of simultaneous rotations of
space and spin, under which $B$ is invariant.
Consequently each $\Psi_m$ is an $H^1$ solution of the zero-mode equation.
In the rotating frame its equation is $H_mu_m=0$, where
\[
 H_m=
 -i\sigma_1\left(\partial_\rho+\frac1{2\rho}\right)
 -i\sigma_3\partial_z
 -c(b_\rho\sigma_1+b_z\sigma_3)
 +\frac m\rho\sigma_2.
\]
The term $1/(2\rho)$ comes from differentiating the spin frame.
It also makes the radial derivative formally skew-adjoint for
the cylindrical measure $\rho\dd\rho\dd z$, so $H_m$ is
formally symmetric for that measure.

\emph{Step 2: the angular identity and its boundary terms.}
First, the distributional identity $\diver e_\rho=1/\rho$,
integration by parts and Cauchy--Schwarz give
\[
 \int_{\R^3}\frac{|\Psi_m|^2}{\rho}\dd x
 \le 2\|\Psi_m\|_2\|\nabla\Psi_m\|_2<\infty.
\]
This is initially proved for smooth compactly supported spinors
and extends to $H^1$ by density.
Since the left-hand side equals
$2\pi\int_{\R}\int_0^\infty|u_m|^2\dd\rho\dd z$,
the unweighted integral of $|u_m|^2$ on the half-plane is finite.

Let $\chi\in\mathcal C_c^\infty((0,\infty)\times\R)$ be real.
Testing $H_mu_m=0$ against $\chi\sigma_2u_m$ and taking twice
the real part gives
\[
 \begin{split}
 0={}&2m\int_{\R}\int_0^\infty
             \chi|u_m|^2\dd\rho\dd z\\
 &+\int_{\R}\int_0^\infty
 \rho\,u_m^*
 \bigl((\partial_\rho\chi)\sigma_3
       -(\partial_z\chi)\sigma_1\bigr)u_m\dd\rho\dd z.
 \end{split}
\]
Here $u_m^*$ denotes the conjugate transpose.
To see the cancellation, all terms of $H_m$ except
$(m/\rho)\sigma_2$ anticommute with $\sigma_2$.
The surviving multiplication term is $2m\chi/\rho$;
differentiating $\chi$ gives the second integral.
The test is legitimate by local $H^1$ approximation away from
$\rho=0$, and no derivative of the bounded coefficients of $B$
is used.

Choose $0\le\chi=\chi_{\eps,R}\le1$ which vanishes for
$\rho<\eps$ and for $|x|>2R$, and equals one when
$\rho>2\eps$ and $|x|<R$.
Choose it as a product of an axis cutoff and a spatial cutoff.
The derivative of the first factor is bounded by $C/\rho$ on
$\{\eps<\rho<2\eps\}$, and that of the second by $C/R$ on the outer annulus.
The absolute value of the second integral is then at most
\[
 C\int_{\{\rho<2\eps\}}\frac{|\Psi_m|^2}{\rho}\dd x
 +\frac{C}{R}\|\Psi_m\|_2^2.
\]
The first term tends to zero as $\eps\to0$ by the weighted
integrability proved above; the second tends to zero as $R\to\infty$.
Dominated convergence in the first integral yields
\[
 0=2m\int_{\R}\int_0^\infty|u_m|^2\dd\rho\dd z.
\]
There is no zero angular frequency, since
$m\in\mathbb Z+\tfrac12$.  Thus every mode vanishes, and
completeness of the angular Fourier decomposition gives $\Psi=0$.
\end{proof}

The absence of an azimuthal component is essential to this argument:
a term $b_\theta e_\theta$ would replace $m/\rho$ by
$m/\rho-cb_\theta$, whose sign need not be fixed; the Hopf contact
geometry of the Loss--Yau potential \cite{LossYau1986} shows that
zero modes do occur in that case.

\begin{proof}[Proof of Theorem~\ref{thm:extremals-no-zero}]
Part~\textup{(a)} is Proposition~\ref{prop:poloidal-no-zero}.  For
part~\textup{(b)} we first treat $A_*$ and $U_*$ themselves.

\emph{Step 3: boundedness of the extremals and the conformal pullback.}
Define
\[
 h(q):=q^{-2}\psi_*(\operatorname{arcosh}(1/q)),
 \qquad 0<q\le1.
\]
Since $q=\sech t\sim2\e^{-t}$ as $t\to\infty$, the endpoint estimates give
$h(0):=C_*/4$ and $h(1)=a_*$.
Thus $h\in \mathcal C([0,1])$, and the preceding endpoint analysis
also gives $G_*\in \mathcal C([0,1])$.
Recalling $s=1+|x|^2$, $\eta=2\rho/s$ and $|N|\le s$, we have
\[
 |A_*|=\frac{4\rho}{s^2}|h(\eta)|,
 \qquad
 |U_*|\le\frac{\|G_*\|_\infty}{s}.
\]
Both potentials are therefore bounded, which proves the
operator-domain assertion in the theorem.
Since $U_*$ is axisymmetric and poloidal,
Proposition~\ref{prop:poloidal-no-zero} already excludes zero modes
for $cU_*$.

For $A_*$, consider the M\"obius map
\[
 M(x)=\frac{(2z,\,1-|x|^2,\,2x_1)}{|x+e_2|^2},
 \qquad
 \lambda(x)=\frac2{|x+e_2|^2},
 \qquad e_2=(0,1,0).
\]
The identity
\[
 M(x)+e_2
 =\frac{2(z,\,x_2+1,\,x_1)}{|x+e_2|^2}
\]
expresses it as an inversion about $-e_2$, followed by a
reflection fixing $-e_2$ and a dilation about the same point.
The two orientation reversals cancel.
Directly,
\[
 M\circ M=\operatorname{id},\qquad
 DM(x)^{\mathsf T}DM(x)=\lambda(x)^2\operatorname{Id},
 \qquad
 \det DM(x)=\lambda(x)^3>0
\]
away from $-e_2$.

A magnetic potential transforms as a one-form.  Its pullback is
therefore $B_M(x)=DM(x)^{\mathsf T}A_*(M(x))$, rather than just
$A_*(M(x))$.
The calculation can be checked from
\[
 s(M(x))=\lambda(x)s(x),\qquad
 \eta(M(x))=\sqrt{1-\eta(x)^2},
\]
and
\[
 DM(x)^{\mathsf T}
 \begin{pmatrix}-M_2(x)\\M_1(x)\\0\end{pmatrix}
 =-\frac{\lambda(x)^2}{2}N(x).
\]
Indeed,
$A_*(y)=4h(\eta(y))s(y)^{-2}(-y_2,y_1,0)$, so
\[
 B_M(x)=-\frac2{s(x)^2}
            h\bigl(\sqrt{1-\eta(x)^2}\bigr)N(x),
 \qquad
 |B_M(x)|\le\frac{2\|h\|_\infty}{s(x)}.
\]
This formula defines a bounded axisymmetric poloidal field on
all of $\R^3$.  In particular, it extends across
$\{\rho=1,\ z=0\}$, where $N=0$, using the continuous extension
of $h$ to zero.

\emph{Step 4: the transformed spinor remains a global zero mode.}
Suppose that $\Psi\in L^2$ is a zero mode for $cA_*$.
Boundedness of $A_*$ gives $\Psi\in H^1$ by the argument of Step~1.
Conformal covariance of the magnetic Dirac operator
\cite[Section~4]{ErdosSolovej2001}, based on the classical spinorial
transformation law \cite[Section~1.4]{Hitchin1974}, \cite{Hijazi1986}
and recalled for zero modes in \cite[Remark~3]{FrankLossSharp2026},
gives
\[
 \Phi(x)=\lambda(x)\beta(x)\Psi(M(x)),
 \qquad
 \mathcal D_{cB_M}\Phi=0
 \quad\text{on }\R^3\setminus\{-e_2\}.
\]
Here $\beta(x)$ is the unitary spinor identification associated
with the orientation-preserving conformal map $M$.
The factor $\lambda$ is the spinor conformal weight in dimension
three.  On compact subsets away from $-e_2$ the transformation
preserves local $H^1$ regularity, and the covariance identity
extends by local $H^1$ approximation.  It requires no
differentiability of the bounded potentials.

Conformal covariance alone does not guarantee that $\Phi\in L^2$.
For this we use $M\circ M=\operatorname{id}$,
$\lambda(M(x))\lambda(x)=1$ and the Jacobian computed above.
Changing variables $y=M(x)$ gives
\[
 \begin{split}
 \int_{\R^3}|\Phi(x)|^2\dd x
 &=\int_{\R^3}\lambda(y)|\Psi(y)|^2\dd y\\
 &=2\int_{\R^3}\frac{|\Psi(y)|^2}{|y+e_2|^2}\dd y
 \le8\int_{\R^3}|\nabla\Psi(y)|^2\dd y<\infty.
 \end{split}
\]
The final step is the three-dimensional Hardy inequality,
applied componentwise.

It remains to remove the exceptional point from the weak equation.
Let $\chi_\eps$ vanish on $B_\eps(-e_2)$ and equal one outside
$B_{2\eps}(-e_2)$, with $|\nabla\chi_\eps|\le C/\eps$.
Then $\|\nabla\chi_\eps\|_2=O(\eps^{1/2})$.
For a compactly supported smooth test spinor $\zeta$, use
$\chi_\eps\zeta$ in the equation away from $-e_2$.
The additional term containing the cutoff derivative is bounded by
\[
 C\|\zeta\|_\infty
   \|\Phi\|_{L^2(B_{2\eps}(-e_2))}
   \|\nabla\chi_\eps\|_2,
\]
and tends to zero.  The other terms converge by local
integrability, since $\Phi\in L^2$ and $B_M$ is bounded.
Thus $\mathcal D_{cB_M}\Phi=0$ on all of $\R^3$ in distributions.
The argument of Step~1 now gives $\Phi\in H^1$, and
Proposition~\ref{prop:poloidal-no-zero} yields $\Phi=0$.
As $\lambda>0$, $\beta$ is unitary and $M$ is invertible away
from one point, it follows that $\Psi=0$.

Finally, the fields $T_{\lambda,z_0}A_*$ and $T_{\lambda,z_0}U_*$ are
obtained by axial translations and the dilations
$A\mapsto\lambda A(\lambda\,\cdot)$, under which
$\Psi\mapsto\lambda^{3/2}\Psi(\lambda\,\cdot)$ maps zero modes to
zero modes.  Hence the conclusion holds for every centre.
\end{proof}

\section{A gauge-canonical inverse}\label{sec:iteration}

We prepare the proof of Theorem~\ref{thm:main-curl} by fixing the
inverse map and establishing its descent and equality properties.

Away from solutions the symbol $\mathfrak D^{-1}$ is ambiguous, because a
curl primitive is determined only modulo $\Wspace$.  We remove the ambiguity
by imposing the nonlinear gauge after each inverse step.
Throughout this section $u\in\Dcurlhalf$.  Put
\[
 \Gamma(u):=u+w(u),
\]
where $w(u)$ is the unique correction from
\eqref{eq:gauge-minimizer}.  Thus $\Gamma(u)\in\Mman$ and
\begin{equation}\label{eq:Gamma-gauge-condition}
 \diver\bigl(\abs{\Gamma(u)}\Gamma(u)\bigr)=0.
\end{equation}
Uniqueness in the homogeneous $L^3$ minimization also gives
\begin{equation}\label{eq:Gamma-homogeneity}
 \Gamma(cu)=c\Gamma(u)\qquad(c\in\R).
\end{equation}
For a divergence-free $F\in L^{3/2}(\R^3,\R^3)$, let
\begin{equation}\label{eq:Biot-Savart-K}
 \mathcal K F:=\curl(-\Delta)^{-1}F.
\end{equation}
The Hardy--Littlewood--Sobolev inequality gives $\mathcal K F\in L^3$, and
$\curl\mathcal K F=F$.  Define the gauge-canonical inverse map by
\begin{equation}\label{eq:canonical-inverse-I}
 I(u):=\Gamma\!\left(\mathcal K
       \bigl(\abs{\Gamma(u)}\Gamma(u)\bigr)\right).
\end{equation}
The source $|\Gamma(u)|\Gamma(u)$ belongs to $L^{3/2}$ and is
divergence-free by \eqref{eq:Gamma-gauge-condition}, so
$I(u)\in\Dcurlhalf$.  It depends only on the class of $u$ modulo $\Wspace$.

\begin{lemma}[The canonical inverse and descent of the quotient]
\label{lem:canonical-inverse-descent}
For $u\in\Dcurlhalf\setminus\Wspace$, the field $I(u)$ defined in
\eqref{eq:canonical-inverse-I} is the unique element of $\Mman$
satisfying
\begin{equation}\label{eq:canonical-inverse-curl}
 \curl I(u)=\abs{\Gamma(u)}\Gamma(u).
\end{equation}
Consequently,
\begin{equation}\label{eq:D-I-identity}
 \mathfrak D(I(u))=\Gamma(u),
 \qquad
 I(cu)=c\abs{c}\,I(u)\quad(c\in\R).
\end{equation}
Moreover,
\begin{equation}\label{eq:inverse-Q-descent}
 Q(I(u))\le Q(u),
\end{equation}
and the decrease in one step is given by an exact formula.  Define
\begin{equation}\label{eq:inverse-alignment-definitions}
 E(u):=\int_{\R^3}\abs{\curl u}^{3/2}\dd x,
 \qquad
 M(u):=\int_{\R^3}\abs{\Gamma(u)}^{3}\dd x,
 \qquad v:=I(u),
\end{equation}
and
\[
 \mathfrak a(u):=
 \frac{M(u)}{E(u)^{2/3}M(v)^{1/3}}\in(0,1].
\]
Then
\begin{equation}\label{eq:inverse-exact-descent-ratio}
 \frac{Q(I(u))}{Q(u)}=\mathfrak a(u)^{3/2}.
\end{equation}
Equality in \eqref{eq:inverse-Q-descent} holds if and only if, for some
$\lambda>0$,
\begin{equation}\label{eq:inverse-equality-cycle}
 \curl u=\lambda\abs{I(u)}I(u),
 \qquad
 I^2(u)=\lambda^{-1}\Gamma(u).
\end{equation}
Equality thus holds exactly on the projective two-cycles of $I$:
two inverse steps return the gauge-minimal field up to a positive factor.
Finally, $I$ interchanges the two cylindrical classes:
\begin{equation}\label{eq:I-type-interchange}
 I(\Tclass)\subset\Sclass,
 \qquad
 I(\Sclass)\subset\Tclass.
\end{equation}
\end{lemma}

\begin{proof}
Equation \eqref{eq:Gamma-gauge-condition} makes
$\abs{\Gamma(u)}\Gamma(u)$ divergence-free, so
\eqref{eq:canonical-inverse-curl} follows from
\eqref{eq:Biot-Savart-K}; applying $\Gamma$ does not change curl.  If two
elements of $\Mman$ have this curl, their difference lies in $\Wspace$, and
both satisfy the Euler equation for the strictly convex $L^3$-minimization in
that affine class.  Hence they coincide.  The two identities in
\eqref{eq:D-I-identity} follow from the definition, uniqueness, and
\eqref{eq:Gamma-homogeneity}.

With the notation in \eqref{eq:inverse-alignment-definitions}, one has
$\Gamma(v)=v$ and
\[
 E(v)=M(u).
\]
Integration by parts and Hölder's inequality give
\begin{align}
 M(u)
 &=\int_{\R^3}\Gamma(u)\cdot\curl v\dd x
  =\int_{\R^3}\curl u\cdot v\dd x \notag\\
 &\le E(u)^{2/3}M(v)^{1/3}.
 \label{eq:inverse-Holder}
\end{align}
The integration-by-parts identity is first proved for smooth compactly
supported fields and then extended by density in the natural
$L^3$--$L^{3/2}$ pairing.
After raising \eqref{eq:inverse-Holder} to the power $3/2$ and rearranging,
we obtain
\[
 Q(v)=\frac{M(u)}{M(v)^{1/2}}
 \le\frac{E(u)}{M(u)^{1/2}}=Q(u).
\]
Division of the two endpoint expressions gives
\[
 \frac{Q(v)}{Q(u)}
 =\frac{M(u)^{3/2}}{E(u)M(v)^{1/2}}
 =\mathfrak a(u)^{3/2},
\]
which proves \eqref{eq:inverse-exact-descent-ratio}.  To identify the
equality case, set
\[
 x:=\frac{v}{M(v)^{1/3}},
 \qquad
 y:=\frac{\mathfrak D(u)}{E(u)^{1/3}}.
\]
Then $\norm{x}{L^3}=\norm{y}{L^3}=1$ and
\[
 \mathfrak a(u)=\int_{\R^3}x\cdot\abs{y}y\dd x.
\]
Equality in Hölder's inequality is therefore equivalent to $x=y$, or,
equivalently, to the first relation in
\eqref{eq:inverse-equality-cycle}.  If that relation holds, then
$\lambda^{-1}\Gamma(u)$ is the unique nonlinear-gauge primitive of
$\abs v v$, and hence equals $I(v)$.  Conversely, the second relation in
\eqref{eq:inverse-equality-cycle} implies the first one after taking curl.
This proves the equality characterization.

Finally, $\mathcal K$ interchanges toroidal and poloidal fields.
To see that the gauge projection preserves each class, let
$R=\operatorname{diag}(1,-1,1)$ act by
$(R\cdot u)(x)=Ru(Rx)$.
Within the $SO(2)$-equivariant space, poloidal fields are fixed by
this reflection and toroidal fields are anti-invariant.
Uniqueness of the gauge minimizer shows that $\Gamma$ commutes with
this action and with rotations; it also commutes with negation by
\eqref{eq:Gamma-homogeneity}.
Thus $\Gamma$ preserves both classes.
The operator $\mathcal K$ reverses reflection parity, by the
transformation rule for curl under orientation-reversing orthogonal
maps.  This proves \eqref{eq:I-type-interchange}.
\end{proof}

\section{Proof of Theorem~\texorpdfstring{\ref{thm:main-curl}}{1.1}}
\label{sec:proof-main}

\subsection{Sharp constants, classification and the first inverse step}
\label{subsec:main-a}

\begin{proof}[Proof of Theorem~\ref{thm:main-curl}\textup{(a)} and \textup{(b)}]
\emph{Reduction and duality.}
The positivity of the global constant and its comparison with the
restricted constants follow from \eqref{eq:constant-chain}.
By Lemma~\ref{lem:T-gauge}, every admissible toroidal field
$A(\psi)$ is already in its $L^3$-minimal gauge.
Proposition~\ref{prop:T-H-correspondence} therefore gives
\[
 Q(A(\psi))=\sqrt{2\pi}\,
 \frac{\displaystyle\int_{\Hh^2}|\nabla_g\psi|_g^{3/2}\dd V_g}
 {\displaystyle\left(\int_{\Hh^2}|\psi|^3\dd V_g\right)^{1/2}}
 \qquad(\psi\ne0).
\]
Taking the infimum over this bijection, and using
\eqref{eq:cylindrical-p-constant-duality} at $p=\widehat p=3/2$,
we obtain
\begin{equation}\label{eq:restricted-constant-equality}
 S_{3/2,\curlname}^{\Tclass}
 =S_{3/2,\curlname}^{\Sclass}
 =\sqrt{2\pi}\,\Lambda_{\Hh^2}.
\end{equation}
The hyperbolic infimum is attained, as recalled in
Section~\ref{sec:hyperbolic}.  Choose a positive minimizer $\psi_*$
normalized by $-\Delta_{3/2,g}\psi_*=\psi_*^2$.
The correspondence above makes $A_*=\rho^{-1}\psi_*e_\theta$
a toroidal extremal, and Proposition~\ref{prop:T-PDE-H} shows that
it solves the full equation \eqref{eq:intro-PDE}, including across
the symmetry axis.  Theorems~\ref{thm:general-p-duality}
and~\ref{thm:general-p-constant-duality} make
$U_*=\mathfrak D(A_*)$ a poloidal extremal solving the same equation,
and give $\mathfrak D(U_*)=A_*$.  These conclusions hold for every
positive hyperbolic minimizer, independently of its centre.

\emph{Classification.}
Let $A=A(\psi)\in\Tclass$ be an extremal of
$S_{3/2,\curlname}^{\Tclass}$.  By the correspondence, $\psi$ minimizes
$\Lambda_{\Hh^2}$, hence has one sign (Section~\ref{sec:hyperbolic}),
and after multiplication by a nonzero constant it is a positive
solution of \eqref{eq:hyperbolic-PDE}.  By
Theorem~\ref{thm:hyperbolic-profile-uniqueness} this solution is
$\psi_*\circ F_{\rho_0,z_0}^{-1}$ for some $\rho_0>0$, $z_0\in\R$, so
that $A=cT_{1/\rho_0,z_0}A_*$ with $c\ne0$; conversely every such field
is an extremal, since $Q$ is invariant under scaling and under the
maps \eqref{eq:euclidean-moves}.  Now let $U\in\Sclass$ be an extremal
of $S_{3/2,\curlname}^{\Sclass}$.  Choose $a>0$ so that
$U':=a\Gamma(U)$ solves \eqref{eq:intro-PDE}; this is possible by
\cite[Theorem~1.4]{MederskiSzulkin2025}.  By
Theorems~\ref{thm:general-p-duality}
and~\ref{thm:general-p-constant-duality}, $\mathfrak D(U')$ is a
toroidal extremal solving \eqref{eq:intro-PDE}, hence of the form
$cT_{\lambda,z_0}A_*$ by the first part; since the maps
\eqref{eq:euclidean-moves} preserve solutions and $cA_*$ solves
\eqref{eq:intro-PDE} only for $|c|=1$, we get
$\mathfrak D(U')=\pm T_{\lambda,z_0}A_*$ and therefore
$U'=\mathfrak D(\mathfrak D(U'))=\pm T_{\lambda,z_0}U_*$, because
$\mathfrak D$ is odd and commutes with the maps
\eqref{eq:euclidean-moves}.  Thus $U=cT_{\lambda,z_0}U_*+w$ with
$c\ne0$ and $w$ curl-free, which proves part~\textup{(b)}.
\end{proof}

\emph{Radial quadrature for the inverse map.}
To evaluate the first inverse step explicitly and to prepare the
convergence argument, we derive the radial formulas.  Let $\psi\ge0$ be a nonzero profile in
$\mathcal D_{\mathrm{rad},O}^{1,3/2}(\Hh^2)$.
Since $\psi^2\in L^{3/2}$ and
$\mathcal D^{1,3/2}(\Hh^2)\hookrightarrow L^3$, strict monotonicity
and coercivity give a unique weak solution
$\chi\in\mathcal D^{1,3/2}(\Hh^2)$ of
\[
 -\Delta_{3/2,g}\chi=\psi^2.
\]
Testing with the negative part gives $\chi\ge0$, and uniqueness
makes $\chi$ radial about the same centre.
Its radial flux $F(t)=\sinh t\,|\chi'|^{-1/2}\chi'$ satisfies
$F'=-\psi^2\sinh t$.
Since the right-hand side is integrable near zero, $F$ has a finite
limit there.  This integration constant is zero: a nonzero flux limit
would make
$\int_0^1|F(t)|^3/\sinh^2t\dd t$ infinite, contrary to finite energy.
Consequently,
\[
 \sinh t\sqrt{-\chi'(t)}
 =\mathcal J_\psi(t),\qquad
 \mathcal J_\psi(t):=\int_0^t\psi(s)^2\sinh s\dd s.
\]
Thus $\chi$ is nonincreasing.  Since $\chi\in L^3(\Hh^2)$, its limit
at infinity is zero, and a second integration gives
\begin{equation}\label{eq:inverse-radial-quadrature}
 \chi(t)=(\mathcal R\psi)(t)
 :=\int_t^\infty\frac{\mathcal J_\psi(s)^2}{\sinh^2s}\dd s,
 \qquad t>0.
\end{equation}
The identities
$\mathfrak D(I^2(A(\psi)))=I(A(\psi))$ and
$\curl I(A(\psi))=|A(\psi)|A(\psi)$, together with
Proposition~\ref{prop:T-PDE-H}, show that the toroidal profile of
$I^2(A(\psi))$ solves the same weak equation as $\chi$.
By uniqueness,
\[
 I^2(A(\psi))=A(\mathcal R\psi).
\]

Write, as radial integrals without the angular factor,
\[
 e(\psi):=\int_0^\infty|\psi'|^{3/2}\sinh t\dd t,\qquad
 m(\psi):=\int_0^\infty\psi^3\sinh t\dd t.
\]
Differentiating \eqref{eq:inverse-radial-quadrature} gives
\[
 e(\mathcal R\psi)
 =\int_0^\infty\frac{\mathcal J_\psi(t)^3}{\sinh^2t}\dd t.
\]
The hyperbolic polar-coordinate formula and the norm identities
\eqref{eq:T-L3-H}--\eqref{eq:T-curl-H} give the first formula below.
The second follows from
$E(I(A(\psi)))=M(A(\psi))$ and
$M(I(A(\psi)))=E(A(\mathcal R\psi))$, with $E$ and $M$ as in
\eqref{eq:inverse-alignment-definitions}:
\begin{equation}\label{eq:inverse-radial-quotients}
 Q(A(\psi))=2\pi\frac{e(\psi)}{m(\psi)^{1/2}},
 \qquad
 Q(I(A(\psi)))=2\pi\frac{m(\psi)}{e(\mathcal R\psi)^{1/2}}.
\end{equation}

\begin{proof}[Proof of Corollary~\ref{cor:explicit-bound}]
The trial field of \cite[Section~6]{BieganowskiMederski2026} is
\[
 A_{\Tclass}(x)
 =\frac{3}{s^2}(-x_2,x_1,0)
 =A(\psi_0),\qquad
 \psi_0(t):=\frac34\sech^2t.
\]
It belongs to $\Tclass\cap\Mman$, and
\[
 Q(A_{\Tclass})=\frac{8\sqrt{10}}7\pi.
\]
For this profile,
\[
 m(\psi_0)=\frac{27}{320},\qquad
 \mathcal J_{\psi_0}(t)=\frac3{16}(1-\sech^3t).
\]
Since $I(A_{\Tclass})=\mathfrak D(A(\mathcal R\psi_0))$,
the calculation leading to \eqref{eq:U-star-intermediate} gives
\[
 I(A_{\Tclass})
 =\frac{\mathcal J_{\psi_0}(t)}{2\rho^2\sinh^2t}N
 =\frac{3}{8s^2}\frac{1-\eta^3}{1-\eta^2}N.
\]
After cancellation, this becomes
\[
 I(A_{\Tclass})(x)
 =\frac{3}{8s^2}\frac{1+\eta+\eta^2}{1+\eta}N(x),
\]
with the quotient extended continuously at $\eta=1$.
The substitution $\eta=\sech t$ gives
\begin{align*}
 e(\mathcal R\psi_0)
 &=\frac{27}{4096}\int_0^1
       \frac{\eta(1-\eta^3)^3}{(1-\eta^2)^{3/2}}\dd\eta\\
 &=\frac{27}{4096}
       \left(\frac{891\pi}{256}-\frac{53}{5}\right).
\end{align*}
For the last equality, integrate by parts and use
\[
 \int_0^1\frac{\eta^2}{\sqrt{1-\eta^2}}\dd\eta=\frac\pi4,\quad
 \int_0^1\frac{\eta^5}{\sqrt{1-\eta^2}}\dd\eta=\frac8{15},\quad
 \int_0^1\frac{\eta^8}{\sqrt{1-\eta^2}}\dd\eta=\frac{35\pi}{256}.
\]
Substitution in \eqref{eq:inverse-radial-quotients} yields
\[
 Q(I(A_{\Tclass}))
 =\frac{96\sqrt{15}\,\pi}{5\sqrt{4455\pi-13568}}.
\]
Its strict comparison with $Q(A_{\Tclass})$ is equivalent to
$\pi>349784/111375$, which follows already from $\pi>223/71$.

It remains to prove that this first iterate is not an extremal.
Put $u_j=I^j(A_{\Tclass})$ and $q_j=Q(u_j)$.
If $q_1$ equalled the common restricted sharp constant,
descent and \eqref{eq:restricted-constant-equality} would give
$q_2=q_1$.  The equality case in
Lemma~\ref{lem:canonical-inverse-descent} would then give
$u_3=c u_1$ for some $c>0$.
Applying $\mathfrak D$ and using
$\mathfrak D(u_{j+1})=u_j$ yields
$u_2=\sqrt c\,u_0$, hence $q_2=q_0$, contradicting $q_1<q_0$.
We have proved the full strict chain
\[
 S_{3/2,\curlname}^{\Tclass}
 <Q(I(A_{\Tclass}))
 =\frac{96\sqrt{15}\,\pi}{5\sqrt{4455\pi-13568}}
 <\frac{8\sqrt{10}}7\pi
 =Q(A_{\Tclass}),
\]
which proves the corollary.
\end{proof}

\subsection{Convergence from every one-signed radial profile}
\label{subsec:main-b}

The proof uses descent, compactness at a fixed centre, and uniqueness
of the positive hyperbolic profile.  The scalar compactness mechanism
is an instance of the abstract inverse-iteration theory of
\cite{Ercole2018}, with $p=3/2$, $q=3$,
$X=\mathcal D^{1,3/2}_{\mathrm{rad},O}(\Hh^2)$ and $Y=L^3(\Hh^2)$ in the
notation of that paper; on the scalar level the scheme is an inverse
power method of the kind introduced for the $p$-Laplacian in
\cite{BiezunerErcoleMartins2009,HyndLindgren2016,HyndLindgren2017}.
We give the full argument here to track the two curl phases and their
different normalizations.

\begin{proof}[Proof of Theorem~\ref{thm:main-curl}\textup{(c)}]
Let $0\ne\psi\in\mathcal D_{\mathrm{rad},O}^{1,3/2}(\Hh^2)$ have one sign,
where $O$ is any fixed hyperbolic centre, and define
\[
 u_0=A(\psi),\qquad u_{n+1}=I(u_n).
\]
The scalar equations, the fixed-centre compactness lemma and the
uniqueness of the positive profile are invariant under hyperbolic
isometries, so there is no restriction on $O$.
Since $I(-u)=-I(u)$, it suffices to assume $\psi\ge0$.
The initial field lies in $\Mman$ by Lemma~\ref{lem:T-gauge}.
Lemma~\ref{lem:canonical-inverse-descent} shows inductively that every
iterate lies in $\Mman$ and is nonzero, because
$\curl u_{n+1}=|u_n|u_n$.
In particular, every quotient below is defined.  Put
\[
 M_n:=\int_{\R^3}|u_n|^3\dd x,\qquad q_n:=Q(u_n).
\]
Lemma~\ref{lem:canonical-inverse-descent} and the equality
\eqref{eq:restricted-constant-equality} proved in part~\textup{(a)} give
\begin{equation}\label{eq:q-lower-bound-limit}
 q_n\to L,\qquad
 L\ge S_{3/2,\curlname}^{\Tclass}
   =S_{3/2,\curlname}^{\Sclass}>0.
\end{equation}
Set $v_n=M_n^{-1/3}u_n$, so that $\|v_n\|_3=1$ and $Q(v_n)=q_n$.
Since $\int|\curl u_{n+1}|^{3/2}=M_n$, their recurrence is
\begin{equation}\label{eq:normalized-vector-recurrence}
 \curl v_{n+1}=q_{n+1}^{2/3}|v_n|v_n.
\end{equation}

The even iterates are toroidal.  Write
\[
 v_{2k}=A(\phi_k)=\frac{\phi_k}{\rho}e_\theta.
\]
The rescaled initial profile $\phi_0=M_0^{-1/3}\psi$ is nonnegative,
nonzero, and radial about $O$.  Suppose inductively that $\phi_k$ has
these properties.  From
\eqref{eq:normalized-vector-recurrence},
\[
 \mathfrak D(v_{2k+2})=q_{2k+2}^{1/3}v_{2k+1}.
\]
Taking curl, using \eqref{eq:normalized-vector-recurrence} once more, and
applying Proposition~\ref{prop:T-PDE-H}, we obtain
\begin{equation}\label{eq:normalized-profile-recurrence-raw}
 -\Delta_{3/2,g}\phi_{k+1}
 =q_{2k+2}^{1/3}q_{2k+1}^{2/3}\phi_k^2
 \qquad\text{in }\Hh^2.
\end{equation}
The $3/2$-Laplace operator on $\mathcal D^{1,3/2}(\Hh^2)$ is strictly
monotone and coercive, so \eqref{eq:normalized-profile-recurrence-raw}
determines $\phi_{k+1}$ uniquely.  Testing with the negative part gives
$\phi_{k+1}\ge0$; the right-hand side is nonnegative, nontrivial and
invariant under the stabilizer of $O$, so by uniqueness $\phi_{k+1}$ is
radial about $O$.  Formula \eqref{eq:inverse-radial-quadrature},
applied to the nonzero nonnegative source in
\eqref{eq:normalized-profile-recurrence-raw}, gives
$\phi_{k+1}(t)>0$ for every $t>0$: the corresponding flux is positive
for all sufficiently large arguments.
This completes the induction.

Put
\[
 f_k:=(2\pi)^{1/3}\phi_k.
\]
By \eqref{eq:T-L3-H}--\eqref{eq:T-curl-H},
\[
 \int_{\Hh^2}f_k^3\dd V_g=1,
 \qquad
 \int_{\Hh^2}\abs{\nabla_gf_k}_g^{3/2}\dd V_g
 =\frac{q_{2k}}{\sqrt{2\pi}}.
\]
Moreover, \eqref{eq:normalized-profile-recurrence-raw} becomes
\begin{equation}\label{eq:f-profile-recurrence}
 -\Delta_{3/2,g}f_{k+1}=\lambda_k f_k^2,
 \qquad
 \lambda_k:=
 \frac{q_{2k+2}^{1/3}q_{2k+1}^{2/3}}{\sqrt{2\pi}}.
\end{equation}
Here and below the equation is understood weakly in
$X:=\mathcal D^{1,3/2}(\Hh^2)$:
\[
 \int_{\Hh^2}\abs{\nabla_gf_{k+1}}_g^{-1/2}
  \ip{\nabla_gf_{k+1}}{\nabla_gh}_g\dd V_g
 =\lambda_k\int_{\Hh^2}f_k^2h\dd V_g
 \qquad(h\in X).
\]
The right-hand side belongs to $X^*$ because $X\hookrightarrow L^3$ and
$f_k^2\in L^{3/2}$.  In view of \eqref{eq:q-lower-bound-limit},
\begin{equation}\label{eq:lambda-ell-limit}
 \lambda_k\to\ell:=\frac{L}{\sqrt{2\pi}},
 \qquad
 \int_{\Hh^2}\abs{\nabla_gf_k}_g^{3/2}\dd V_g
 \to\ell.
\end{equation}

The sequence $(f_k)$ is bounded in
$\mathcal D_{\mathrm{rad},O}^{1,3/2}(\Hh^2)$.  By
Lemma~\ref{lem:fixed-centre-radial-compactness}, after passing to a
subsequence, not relabelled, we may assume that, for some $f,g$,
\begin{align}
 f_k&\rightharpoonup f,
 &f_{k+1}&\rightharpoonup g
 &&\text{in }\mathcal D^{1,3/2}(\Hh^2),
 \notag\\
 f_k&\to f,
 &f_{k+1}&\to g
 &&\text{in }L^3(\Hh^2).
 \label{eq:profile-subsequence-convergence}
\end{align}
In particular, $f,g\ge0$ and $\norm{f}{L^3}=\norm{g}{L^3}=1$.
Testing \eqref{eq:f-profile-recurrence} by $f_{k+1}$ gives
\[
 \int_{\Hh^2}\abs{\nabla_gf_{k+1}}_g^{3/2}\dd V_g
 =\lambda_k\int_{\Hh^2}f_k^2f_{k+1}\dd V_g.
\]
Using \eqref{eq:lambda-ell-limit} and
\eqref{eq:profile-subsequence-convergence}, we obtain
\[
 \int_{\Hh^2}f^2g\dd V_g=1.
\]
Equality in Hölder's inequality, together with the two unit $L^3$ norms,
implies $f=g$.

Let $\mathcal A=-\Delta_{3/2,g}:X\to X^*$.
By Browder--Minty there is a unique $h\in X$ such that
$\mathcal Ah=\ell f^2$.
On bounded subsets of $X$, the monotonicity inequality for $p=3/2$
gives
\begin{align*}
 \|f_{k+1}-h\|_X^2
 &\le C\langle\mathcal Af_{k+1}-\mathcal Ah,f_{k+1}-h\rangle\\
 &\le C\|\lambda_k f_k^2-\ell f^2\|_{X^*}
          \|f_{k+1}-h\|_X,
\end{align*}
where $\|h\|_X=\|\nabla_gh\|_{3/2}$.
The last dual norm tends to zero because $f_k\to f$ in $L^3$.
Thus $f_{k+1}\to h$ strongly in $X$.
Its $L^3$ limit is $g=f$, so $h=f$ and
\[
 -\Delta_{3/2,g}f=\ell f^2,\qquad \|f\|_{L^3}=1.
\]
Testing with $f$ yields $\|\nabla_gf\|_{3/2}^{3/2}=\ell$.
The energies of $f_k$ tend to the same value by
\eqref{eq:lambda-ell-limit}; weak convergence and uniform convexity
therefore imply $f_k\to f$ strongly in $X$ along this subsequence.
The radial representation \eqref{eq:inverse-radial-quadrature}, with
source $\ell f^2$, gives $f(t)>0$ for $t>0$: since $f\not\equiv0$,
the source flux is positive for all sufficiently large radii, so every
tail integral in that representation is positive.  After the rescaling
$\widetilde f=\ell^{2/3}f$, Theorem~\ref{thm:hyperbolic-profile-uniqueness}
identifies $\widetilde f$ with the unique positive finite-energy profile
centred at $O$.  Let $f_*$ be the positive hyperbolic minimizer centred at $O$ with
$\norm{f_*}{L^3}=1$.  Its Euler equation is
\[
 -\Delta_{3/2,g}f_*=\Lambda_{\Hh^2}f_*^2.
\]
Thus $\Lambda_{\Hh^2}^{2/3}f_*$ is another centred positive solution of the
coefficient-one equation.  Uniqueness gives
\[
 \ell^{2/3}f=\Lambda_{\Hh^2}^{2/3}f_*.
\]
Taking $L^3$ norms therefore yields
\begin{equation}\label{eq:ell-equals-Lambda}
 \ell=\Lambda_{\Hh^2}.
\end{equation}
The same argument applies to every subsequence, and the limit is always the
same profile of $L^3$ norm one; hence the whole sequence of rescaled even
profiles converges strongly in $\mathcal D^{1,3/2}(\Hh^2)$ and in
$L^3(\Hh^2)$.  Combining \eqref{eq:lambda-ell-limit},
\eqref{eq:ell-equals-Lambda} and \eqref{eq:T-constant-H} we conclude that
\begin{equation}\label{eq:inverse-iteration-limit}
 Q(u_n)\to L=\sqrt{2\pi}\Lambda_{\Hh^2}
  =S_{3/2,\curlname}^{\Tclass}.
\end{equation}
Although only the even iterates entered the compactness argument, $L$ was
defined in \eqref{eq:q-lower-bound-limit} as the limit of the whole monotone
sequence $(q_n)$, so \eqref{eq:inverse-iteration-limit} holds for the full
sequence.
If $u_{2k}=A(\psi_k)$, the profiles used above satisfy
$f_k=\psi_k/\|\psi_k\|_{L^3(\Hh^2)}$, and
\eqref{eq:T-L3-H}--\eqref{eq:T-curl-H} turn the strong convergence
$f_k\to f_*$ in $X$ and in $L^3(\Hh^2)$ into
\[
 v_{2k}=(2\pi)^{-1/3}A(f_k)\longrightarrow(2\pi)^{-1/3}A(f_*)
 =\frac{A_*}{\|A_*\|_{L^3}}
\]
in $L^3(\R^3)$, with curls converging in $L^{3/2}(\R^3)$; here
$O=(1,0)$, and for a general centre $A_*$ is replaced by its image
under \eqref{eq:euclidean-moves}.

\emph{The odd iterates.}
The map $I$ is continuous for the $L^3$ topology: $u\mapsto|u|u$ is
continuous from $L^3$ to $L^{3/2}$, $\mathcal K$ is bounded from
$L^{3/2}$ to $L^3$, and $\Gamma=\mathrm{Id}-P_{\Wspace}$, where
$P_{\Wspace}$ is the metric projection onto the closed subspace
$\Wspace$ of the uniformly convex space $L^3$, is continuous.  With
$M_*:=\int_{\R^3}|A_*|^3\dd x$ and $v_\infty:=M_*^{-1/3}A_*$, the
definition of $I$ and \eqref{eq:Gamma-homogeneity} give
\[
 v_{2k+1}=\frac{I(u_{2k})}{M_{2k+1}^{1/3}}
 =\frac{M_{2k}^{2/3}}{M_{2k+1}^{1/3}}\,
   \Gamma\mathcal K\bigl(|v_{2k}|v_{2k}\bigr)
 =q_{2k+1}^{2/3}\,\Gamma\mathcal K\bigl(|v_{2k}|v_{2k}\bigr)
 \longrightarrow L^{2/3}\,\Gamma\mathcal K\bigl(|v_\infty|v_\infty\bigr)
\]
in $L^3(\R^3)$.  Since $A_*$ solves \eqref{eq:intro-PDE}, the
gauge-minimal field with curl $|A_*|A_*$ is $U_*=\mathfrak D(A_*)$, so
$\Gamma\mathcal K(|v_\infty|v_\infty)=M_*^{-2/3}U_*$; moreover
$L=Q(A_*)=M_*^{1/2}$ and $\int_{\R^3}|U_*|^3\dd x=M_*$ by the Nehari
identity.  Hence $v_{2k+1}\to M_*^{-1/3}U_*=U_*/\|U_*\|_{L^3}$, and
$\curl v_{2k+1}=q_{2k+1}^{2/3}|v_{2k}|v_{2k}\to L^{2/3}|v_\infty|v_\infty
=\curl(M_*^{-1/3}U_*)$ in $L^{3/2}(\R^3)$.  For nonpositive initial
data the oddness of $I$ reverses the signs of both limits.
\end{proof}

\subsection{Breakdown of the forward iteration}
\label{subsec:main-c}

\begin{proposition}[Breakdown of the forward iteration]
\label{prop:forward-breakdown}
Let $u_j:=\mathfrak D^j(A_{\Tclass})$ and $q_j:=Q(u_j)$, as long as
the forward image belongs to the natural domain.  Then
\begin{equation}\label{eq:forward-domain-breakdown}
 u_j\in\Dcurlhalf\quad(j=0,1,2),
 \qquad
 u_3\in L^3(\R^3,\R^3)\setminus\Dcurlhalf,
\end{equation}
and $q_0<q_1<q_2$, with the values
\eqref{eq:forward-q0}--\eqref{eq:forward-q2} below.
\end{proposition}

\begin{proof}
The involution in Theorem~\ref{thm:general-p-duality} at $p=3/2$ acts
on solutions; it does
not define a dynamical system on the energy space.  For a toroidal field
with a hyperbolically radial profile, Proposition~\ref{prop:T-PDE-H} gives, whenever the expressions
belong to the required spaces,
\[
 \mathfrak D^2(A(\psi))
 =A\!\left(\operatorname{sgn}(F)\sqrt{\abs F}\right),
 \qquad
 F:=-\Delta_{3/2,g}\psi.
\]
Indeed,
\[
 \curl\mathfrak D(A(\psi))=\frac{F}{\rho^2}e_\theta.
\]

The first three finite quotients are
\begin{align}
 q_0&=\frac{8\sqrt{10}}7\pi
      \approx11.3538151609098,
 \label{eq:forward-q0}\\
 q_1&=
 \pi\sqrt{\frac72}
 \left(\frac32\right)^{3/2}
 \mathrm B\left(\frac58,\frac74\right)
 \approx11.6468871437934,
 \notag\\
 q_2&=
 \frac{\pi}{4\sqrt2}\,
 \frac{\mathcal H_0}
 {\mathrm B(\frac58,\frac74)^{1/2}}
 \approx18.5815053070815,
 \label{eq:forward-q2}
\end{align}
where
\begin{align*}
 \mathcal H_0
 &:={}
 \int_0^1x^{-15/16}(1-x)^{-3/8}(1+5x)^{3/2}\dd x
 \\
 &=
 \mathrm B\left(\frac1{16},\frac58\right)
 {}_2F_1\left(-\frac32,\frac1{16};\frac{11}{16};-5\right).
\end{align*}
In particular $q_0<q_1<q_2$, and $q_3$ is not defined.

For $\psi_0=(3/4)\sech^2t$, a direct radial calculation gives
\[
 F_0:=-\Delta_{3/2,g}\psi_0
 =\left(\frac32\right)^{3/2}
 \sech^3t\,\tanh^{-1/2}t.
\]
Since $\sinh t\sqrt{-\psi_0'(t)}=\sqrt{3/2}\,\tanh^{3/2}t\to0$ as
$t\to0^+$, this formula holds distributionally without a point mass at
the hyperbolic centre.

Consequently,
\[
 u_2=A(\chi_0),
 \qquad
 \chi_0(t):=\sqrt{F_0(t)}
 =\left(\frac32\right)^{3/4}
 \sech^{3/2}t\,\tanh^{-1/4}t.
\]

Put $M_j:=\int_{\R^3}\abs{u_j}^3\dd x$ and
$E_j:=\int_{\R^3}\abs{\curl u_j}^{3/2}\dd x$, so that
$q_j=E_j/M_j^{1/2}$ and $M_{j+1}=E_j$.  In radial hyperbolic
coordinates, the identity
$\int_0^\infty\sech^at\,\tanh^bt\,\sinh t\dd t
 =\frac12\mathrm B\bigl(1+\frac b2,\frac{a-1}2\bigr)$, valid for $a>1$
and $b>-2$, gives $M_0=27\pi^2/80$, $E_0=\frac{8\pi^2}7(3/2)^{3/2}$ and
$E_1=2\pi^2(3/2)^{9/4}\mathrm B(\frac58,\frac74)$, while
$-\chi_0'(t)=\frac14(3/2)^{3/4}\sech^{3/2}t\,\tanh^{-5/4}t\,(1+5\tanh^2t)$
and the substitution $x=\tanh^2t$ give
$E_2=\frac{\pi^2}4(3/2)^{9/8}\mathcal H_0$.  Formulas
\eqref{eq:forward-q0}--\eqref{eq:forward-q2} follow.

Every forward image $u_{j+1}$ which belongs to $L^3$ is automatically in the
nonlinear gauge, because
$\abs{u_{j+1}}u_{j+1}=\curl u_j$.  Since $u_0=A_{\Tclass}$ is also in that
gauge, whenever two consecutive
forward fields belong to the energy space,
\[
 I(\mathfrak D(u_j))=u_j.
\]
Lemma~\ref{lem:canonical-inverse-descent} gives
$Q(u_j)\le Q(u_{j+1})$.  Equality in the first step would require $F_0$ to
be a constant multiple of $\psi_0^2$, which is false.  Equality in the second
would imply, by taking curls and using uniqueness of the toroidal gauge
representative, that $\chi_0$ is a constant multiple of $\psi_0$, which is
also false.  Thus the two inequalities are strict.

It remains to determine the domain of the third image.  As $t\to0^+$,
\[
 \chi_0(t)
 =\left(\frac32\right)^{3/4}t^{-1/4}(1+O(t^2)),
 \qquad
 -\chi_0'(t)
 =\frac14\left(\frac32\right)^{3/4}t^{-5/4}(1+O(t^2)).
\]
Moreover,
\[
 \sinh t\sqrt{-\chi_0'(t)}
 \sim\frac12\left(\frac32\right)^{3/8}t^{3/8}
 \to0.
\]
Thus the flux through small circles around the hyperbolic centre tends to
zero and the distributional $3/2$-Laplacian of $\chi_0$ carries no Dirac
mass there.  Moreover $\chi_0\in L^3(\Hh^2)$ and
\[
 \abs{\chi_0'}^{3/2}\sinh t\asymp t^{-7/8}
\]
is integrable at zero, which confirms $u_2\in\Dcurlhalf$.  On the other hand,
if $F_2:=-\Delta_{3/2,g}\chi_0$, then
\[
 F_2(t)
 =\frac1{\sinh t}\frac{\dd}{\dd t}
 \left(\sinh t\sqrt{-\chi_0'(t)}\right)
 =\frac3{16}\left(\frac32\right)^{3/8}
 t^{-13/8}(1+o(1)).
\]
Consequently,
\[
 \abs{F_2(t)}^{3/2}\sinh t\asymp t^{-23/16},
\]
which is not integrable at zero.
At infinity, $\chi_0(t)=O(\e^{-3t/2})$ and
$F_2(t)=O(\e^{-3t/4})$, so the relevant $L^3$ and curl-energy densities
are integrable there.
If $\curl u_3$ belonged to $L^{3/2}$, it would coincide
almost everywhere with the pointwise curl $\rho^{-2}F_2e_\theta$, whose
$L^{3/2}$ norm is $(2\pi\int_{\Hh^2}\abs{F_2}^{3/2}\dd V_g)^{2/3}=\infty$.
On the other hand, $\int_{\R^3}\abs{u_3}^3\dd x=E_2<\infty$.  This proves
\eqref{eq:forward-domain-breakdown}.
\end{proof}

\section{Numerical values}\label{sec:numerics}

The numerical values in this section are given for orientation.  They are
not certified enclosures and are not used in any proof.

For a decreasing radial solution define the flux
$\mathcal J(t):=\sinh t\sqrt{-\psi_*'(t)}$.  Equation
\eqref{eq:hyperbolic-radial-ODE} becomes the first-order flux system
\[
 \psi_*'(t)=-\Bigl(\frac{\mathcal J(t)}{\sinh t}\Bigr)^2,
 \qquad
 \mathcal J'(t)=\psi_*(t)^2\sinh t,
\]
with $\psi_*(0)=a_*>0$ and $\mathcal J(0)=0$.
The local expansion from Remark~\ref{rem:profile-regularity} gives
$\mathcal J(t)=\frac{a_*^2}{2}t^2+o(t^2)$ and
$\psi_*(t)=a_*-\frac{a_*^4}{12}t^3+o(t^3)$ as $t\to0^+$.
Numerical shooting for the rapidly decaying branch, using fourth-order
Runge--Kutta integration and bisection in the initial value, gives
$a_*=\psi_*(0)\approx2.18885994$.

Let $I_*:=\int_0^\infty\abs{\psi_*'}^{3/2}\sinh t\dd t
=\int_0^\infty\psi_*^3\sinh t\dd t$, where the equality follows by
testing \eqref{eq:hyperbolic-radial-ODE} with $\psi_*$; numerically,
$I_*\approx3.21240637$.
The full hyperbolic integrals contain an additional angular factor
$2\pi$, so Proposition~\ref{prop:T-H-correspondence} gives
$S_{3/2,\curlname}^{\Tclass}=S_{3/2,\curlname}^{\Sclass}
=2\pi\sqrt{I_*}\approx11.26147061$.
The improvement over the original explicit bound is
$\frac{8\sqrt{10}}7\pi-11.2614706054\approx0.0923$, that is, about
$0.8\%$.

The inverse iteration of Theorem~\ref{thm:main-curl}\textup{(c)},
started from $A_{\Tclass}$ and computed by the quadratures
\eqref{eq:inverse-radial-quadrature} and
\eqref{eq:inverse-radial-quotients} on a fine grid, gives the
following values of $q_n=Q(I^n(A_{\Tclass}))$:
\[
\begin{array}{c|ccccccccc}
 n & 0 & 1 & 2 & 3 & 4 & 6 & 8 & 12 & 24\\ \hline
 q_n & 11.35382 & 11.29481 & 11.27676 & 11.26990 & 11.26662
     & 11.26366 & 11.26244 & 11.26167 & 11.26147
\end{array}
\]
The differences $q_n-S_{3/2,\curlname}^{\Tclass}$ decrease eventually by a
factor of about $0.67$ per step, which suggests geometric convergence of the
projective iteration; the convergence theorem does not provide a rate.

\section{Open problems}\label{sec:open-problems}

The main question left open is whether the restricted extremals are
global minimizers.

\begin{conjecture}[Symmetry of global minimizers]
Every global minimizer of $S_{3/2,\curlname}$ is, after a conformal
transformation of $\R^3$ and a change of gauge, a field of class
$\Sclass$, equivalently of class $\Tclass$; in particular
$S_{3/2,\curlname}=S_{3/2,\curlname}^{\Sclass}$ and $U_*$ is a global
ground state.
\end{conjecture}

The evidence is as follows.  The best values known to us are
$Q(A_{\mathrm{LY}})=4\pi\approx12.566$, $Q(A_{\Tclass})\approx11.354$,
and $S_{3/2,\curlname}^{\Tclass}\approx11.261$, and no competitor
outside the two classes with a smaller quotient is known; the
exclusion of the Loss--Yau candidate within the axisymmetric class
\cite{BieganowskiMederski2026} does not decide the question.  On
$\mathbb S^3$ the centred extremals $A_*$ and $U_*$ correspond to
one-forms along the two circle factors of the Hopf tori, so the
conjecture asks whether a global
minimizer can be moved by conformal maps onto a single Hopf circle
factor.

\begin{question}[Symmetry and dual exponents]
	For $p\ne3/2$, which of $S_{p,\curlname}^{\Tclass}$ and
	$S_{p,\curlname}^{\Sclass}$ is smaller?  The duality identities relate
	opposite classes at $p$ and $\widehat p$ and do not determine their
	ordering at a fixed exponent.
	Does either restricted class attain the unrestricted constant at $p=2$ or
	at $p=6/5$?
\end{question}

\begin{question}[Rate and basin of the inverse iteration]
Theorem~\ref{thm:main-curl}\textup{(c)} gives convergence of the
normalized iterates without a rate.  Is the convergence geometric,
that is, are there $C>0$ and $\lambda<1$ such that
$\|v_n-v_\infty^{(n)}\|_{L^3}\le C\lambda^n$ and
$Q(u_n)-S_{3/2,\curlname}^{\Tclass}\le C\lambda^n$,
where $v_\infty^{(n)}$ denotes the limit of the even or of the odd
iterates according to the parity of $n$?  The values in
Section~\ref{sec:numerics} suggest that
$Q(u_n)-S_{3/2,\curlname}^{\Tclass}$ decreases by a factor of about
$0.67$ per step.  The second question concerns the
\emph{basin of attraction} of the two-cycle, that is, the set of
initial fields $u_0$ whose normalized iterates $v_n$ converge to
$\pm T_{\lambda,z_0}A_*/\|A_*\|_{L^3}$ and
$\pm T_{\lambda,z_0}U_*/\|U_*\|_{L^3}$ for some centre. 
\end{question}

\section*{Acknowledgments}
The author was partly supported by the National Science Centre, Poland,\\
grant no.~2023/51/B/ST1/00968.

\section*{AI statement}
The author used AI tools to assist with language editing, the presentation of mathematical arguments developed by the author, literature searches, and the numerical computations reported in Sections ~\ref{subsec:main-c} and~\ref{sec:numerics}. All mathematical arguments and proofs in the final manuscript were written and verified by the author.

\begingroup
\sloppy
\raggedright
\raggedbottom

\endgroup

\end{document}